\documentclass[11pt]{amsart}
\usepackage{amsmath,amssymb}
\usepackage[dvips]{epsfig}
\usepackage{graphicx}
\usepackage{pst-plot, pstricks, pst-node}
\usepackage{amsmath}
\usepackage{amsfonts}
\usepackage{amssymb}
\usepackage{graphicx}%
\providecommand{\U}[1]{\protect\rule{.1in}{.1in}}
\newtheorem{theorem}{Theorem}

\newtheorem{definition}{Definition}
\newtheorem{lemma}{Lemma}
\newtheorem{proposition}{Proposition}
\newtheorem{remark}{Remark}

\makeatletter\@addtoreset {equation}{section}\makeatother

\begin{document}

\title[Bifurcations of multi-vortex configurations]{\textbf{Bifurcations of multi-vortex configurations \\in rotating Bose--Einstein condensates}}

\author[C. Garc\'{\i}a--Azpeitia and D.E. Pelinovsky]{C. Garc\'{\i}a--Azpeitia$^{1}$ and D.E. Pelinovsky$^{2}$}

\address{ $^{1}$ Departamento de Matem\'{a}ticas, Facultad de Ciencias, Universidad Nacional Aut\'{o}noma de M\'{e}xico, Mexico City, Distrito Federal, Mexico, 04510} \email{cgazpe@ciencias.unam.mx}

\address{$^{2}$ Department of Mathematics and Statistics, McMaster University, Hamilton, Ontario, Canada, L8S 4K1} \email{dmpeli@math.mcmaster.ca}

\date{\today}
\maketitle

\begin{abstract}
We analyze global bifurcations along the family of radially symmetric vortices
in the Gross--Pitaevskii equation with a symmetric harmonic potential and a
chemical potential $\mu$ under the steady rotation with frequency $\Omega$.
The families are constructed in the small-amplitude limit when the chemical
potential $\mu$ is close to an eigenvalue of the Schr\"{o}dinger operator for
a quantum harmonic oscillator. We show that for $\Omega$ near $0$, the Hessian
operator at the radially symmetric vortex of charge $m_{0}\in\mathbb{N}$ has
$m_{0}(m_{0}+1)/2$ pairs of negative eigenvalues.
When the parameter $\Omega$ is increased, $1+m_{0}(m_{0}-1)/2$ global bifurcations happen.
Each bifurcation results in the disappearance of a pair of negative eigenvalues in
the Hessian operator at the radially symmetric vortex. The distributions of
vortices in the bifurcating families are analyzed by using symmetries of the
Gross--Pitaevskii equation and the zeros of Hermite--Gauss eigenfunctions.
The vortex configurations that can be found in the bifurcating families are the asymmetric
vortex $(m_0 = 1)$, the asymmetric vortex pair $(m_0 = 2)$,
and the vortex polygons $(m_0 \geq 2)$.

\noindent\textbf{Keywords:} Gross--Pitaevskii equation, rotating vortices, harmonic
potentials, Lyapunov--Schmidt reductions, bifurcations and symmetries.
\end{abstract}

\section{Introduction}

This work addresses the Gross-Pitaevskii equation describing rotating
Bose-Einstein condensates (BEC) placed in a symmetric harmonic trap. It is now
well established from the energy minimization methods that vortex
configurations become energetically favorable for larger rotating frequencies
(see review \cite{review} for physics arguments). Ignat and Millot
\cite{IM1,IM2} confirmed that the vortex of charge one near the center of
symmetry is a global minimizer of energy for a frequency above the first
critical value. Seiringer \cite{Seiringer} proved that a vortex configuration
with charge $m_{0}$ becomes energetically favorable to a vortex configuration
with charge $(m_{0}-1)$ for a frequency above the $m_{0}$-th critical value
and that radially symmetrically vortices of charge $m_{0} \geq 2$ cannot be
global minimizers of energy. The questions on how the $m_{0}$ individual
vortices of charge one are placed near the center of symmetry to form an
energy minimizer remain open since the time of \cite{IM1,IM2,Seiringer}.

For the vortex of charge one, it is shown by using variational approximations
\cite{Castin} and bifurcation methods \cite{PeKe13} that the construction of
energy minimizers is not trivial past the threshold value for the rotation
frequency, where the radially symmetric vortex becomes a local minimizer of
energy\footnote{The threshold value of the rotation frequency for the
bifurcation of local minimizers in \cite{Castin,PeKe13} is smaller than the
first critical value in \cite{IM1,IM2}, at which the charge-one vortex becomes
the global minimizer of energy.}. Namely, in addition to the radially
symmetric vortex, which exists for all rotation frequencies, there exists
another branch of the asymmetric vortex solutions above the threshold value,
which are represented by a vortex of charge one displaced from the center of
rotating symmetric trap. The distance from the center of the harmonic trap
increases with respect to the detuning rotation frequency above the threshold
value, whereas the angle is a free parameter of the asymmetric vortex
solutions. Although the asymmetric vortex is not a local energy minimizer, it
is nevertheless a constrained energy minimizer, for which the constraint
eliminates the rotational degree of freedom and defines the angle of the
solution family uniquely. Consequently, both radially symmetric and asymmetric
vortices are orbitally stable in the time evolution of the
Gross--Pitaevskii equation for the rotating frequency slightly above the
threshold value \cite{PeKe13}.

Further results on the stability of equilibrium configurations of several
vortices of charge one in rotating harmonic traps were found numerically, from
the predictions given by the finite-dimensional system for dynamics of
individual vortices \cite{theo2,theo1,Navarro}. The two-vortex equilibrium
configuration arises again above the threshold value for the rotation
frequency with the two vortices of charge one being located symmetrically with
respect to the center of the harmonic trap. However, the symmetric vortex pair
is stable only for small distances from the center and it losses stability for
larger distances. Once it becomes unstable, another asymmetric pair of two
vortices bifurcate, where one vortex has a smaller-than-critical distance from
the center and the other vortex has a larger-than-critical distance from the
center. The asymmetric pair is stable in numerical simulations and coexist for
rotating frequencies above the threshold value with the stable symmetric
vortex pair located at the smaller-than-critical distances \cite{Navarro}. The
symmetric pair is a local minimizer of energy above the threshold value,
whereas the asymmetric pair is a local constrained minimizer of energy, where
the constraint again eliminates the rotational degree of freedom \cite{KevPelNew}.

This work continues analysis of local bifurcations of vortex configurations in
the Gross--Pitaevskii (GP) equation with a cubic repulsive interaction and a
symmetric harmonic trap. In a steadily rotating frame with the rotation
frequency $\Omega$, the main model can be written in the normalized form
\begin{equation}
i u_{t} = -(\partial_{x}^{2} + \partial_{y}^{2}) u + (x^{2}+y^{2}) u +
\left\vert u\right\vert ^{2}u + i \Omega(x \partial_{y} - y \partial_{x}) u,
\quad(x,y) \in\mathbb{R}^{2}. \label{Ec}%
\end{equation}
The associated energy of the GP equation is given by
\begin{equation}
E(u) = \int\!\int_{\mathbb{R}^{2}} \left[  |\nabla u|^{2} + |x|^{2} |u|^{2} +
\frac{1}{2} |u|^{4} + \frac{i}{2} \Omega\bar{u}(x \partial_{y} - y
\partial_{x}) u - \frac{i}{2} \Omega u (x \partial_{y} - y \partial_{x})
\bar{u} \right]  dx dy. \label{Energy}%
\end{equation}
Compared to work in \cite{PeKe13}, we do not use the scaling for the
semi-classical limit of the GP equation and parameterize the vortex solutions
in terms of the chemical potential $\mu$ arising in the separation of
variables $u(t,x,y) = e^{-i \mu t} U(x,y)$. The profile $U$ satisfies the
stationary GP equation in the form
\begin{equation}
\mu U = -(\partial_{x}^{2} + \partial_{y}^{2}) U + (x^{2}+y^{2}) U +
\left\vert U\right\vert ^{2} U + i \Omega(x \partial_{y} - y \partial_{x}) U,
\quad(x,y) \in\mathbb{R}^{2}. \label{statGP}%
\end{equation}

Local bifurcations of small-amplitude vortex solutions in the GP equation
(\ref{Ec}) have been addressed recently in many publications. We refer to
these small-amplitude vortex solutions as the \emph{primary branches}.
Classification of localized (soliton and vortex) solutions from the triple
eigenvalue was constructed by Kapitula \emph{et al.} \cite{Kap1} with the
Lyapunov--Schmidt reduction method. Existence, stability, and bifurcations of
radially symmetric vortices with charge $m_{0}\in\mathbb{N}$ were studied by
Kollar and Pego \cite{Kollar} with shooting methods and Evans function
computations. Symmetries of nonlinear terms were used to continue families of
general vortex and dipole solutions from the linear limit by Contreras and
Garc\'{\i}a-Azpeitia \cite{CoGa15} by using equivariant degree theory
\cite{IzVi03} and bifurcation methods \cite{GaIz12}. Existence and stability of
stationary states were analyzed in \cite{Thomann,Hani} with the amplitude equations
for the Hermite function decompositions and their truncation at the continuous
resonant equation. Vortex dipoles were studied with normal form
equations and numerical approximations in \cite{GKC}. Numerical evidences of
existence, bifurcations, and stability of such vortex and dipole solutions can
be found in a vast literature \cite{21,34,38,PeKe11,57}.

Compared to the previous literature, our results will explore the recent
discovery of \cite{PeKe13} of how bifurcations of unconstrained and
constrained minimizers of energy are related to the spectral stability problem
of radially symmetric vortices in the small-amplitude limit, in particular,
with the eigenvalues of negative Krein signature which are known to
destabilize dynamics of vortices \cite{Kollar}. Therefore, we consider
bifurcations of \emph{secondary branches} of multi-vortex solutions from the
primary branch of the radially symmetric vortex of charge $m_{0}\in\mathbb{N}%
$. The primary branch is parameterized by only one parameter $\omega
:=\mu+m_{0}\Omega$ in the small-amplitude limit, whereas the secondary
branches of multi-vortex configurations are parameterized by two parameters
$\omega$ and $\Omega$.

As a particular example with $m_0 = 2$, we show that the asymmetric pair of two vortices
of charge one bifurcates from the radially symmetric vortex of charge two for $\Omega$ below
but near $\Omega_{0}=2$. Similarly to the symmetric charge-two vortex 
\cite{Kap1,Kollar}, the asymmetric pair of two charge-one vortices is born
unstable but it is more energetically favorable near the bifurcation threshold
compared to the charge-two vortex in the case of no rotation ($\Omega=0$). If
the charge-two vortex is a saddle point of the energy $E$ in (\ref{Energy})
with three pairs of negative eigenvalues for $\Omega=0$, it has only one pair
of negative eigenvalues for $\Omega$ below but near $\Omega_{0}=2$.

We note that the bifurcation technique developed here is not feasible by the
methods developed in \cite{Kap1} because of the infinitely many resonances at
$\Omega_{0} = 2$. Nevertheless, we show that these resonances are avoided for
$\Omega$ below but near $\Omega_{0} = 2$.

For a charge-one vortex, a similar bifurcation happens for $\Omega$ below but
near $\Omega_{0}=2$, which has been already described in \cite{PeKe13} in
other notations and with somewhat formal analysis. The results developed here
allows us to give a full justification of the results of \cite{PeKe13} for a
charge-one vortex, but also to extend the analysis to the charge-two vortex, 
as well as to a radially symmetric vortex of a general charge
$m_{0}\in\mathbb{N}$.

We also consider all other secondary bifurcations of the radially symmetric
vortices of charge $m_{0}\in\mathbb{N}$ when the frequency parameter $\Omega$
is increased from zero in the interval $(0,2)$. We show that each bifurcation
results in the disappearance of a single pair of negative eigenvalues in the
characterization of radially symmetric vortices as saddle points of the energy
$E$ in (\ref{Energy}).

As a particular example, we show that the symmetric charge-two vortex has a
bifurcation at $\Omega$ near $\Omega_{*} = 2/3$, where another secondary
branch bifurcates. The new branch contains three charge-one vortices at the
vertices of an equilateral triangle and a vortex of anti-charge one
at the center of symmetry. Again, the secondary branch inherits
instability of the radially symmetric vortex along the primary branch in the
small-amplitude limit. Past the bifurcation point, the radially symmetric
vortex of charge two has two pairs of negative eigenvalues. The bifurcation
result near $\Omega_{*} = 2/3$ was not obtained in the previous work
\cite{Kap1}.

In the case of the multi-vortex configurations of the total charge two, we can
conjecture that the local minimizers of energy given by the symmetric pair of
two charge-one vortices as in \cite{Navarro} can be found from a
\emph{tertiary bifurcation} along the secondary branch given by the asymmetric
pair of charge-one vortices. However, it becomes technically involved to
approximate the secondary branch near the bifurcation point and to find the
tertiary bifurcation point.

The following theorem represents the main result of our
paper. A schematic illustration is given on Figure \ref{fig-scheme}.

\begin{figure}[th]

\begin{center}
\resizebox{13.0cm}{!}{ \includegraphics{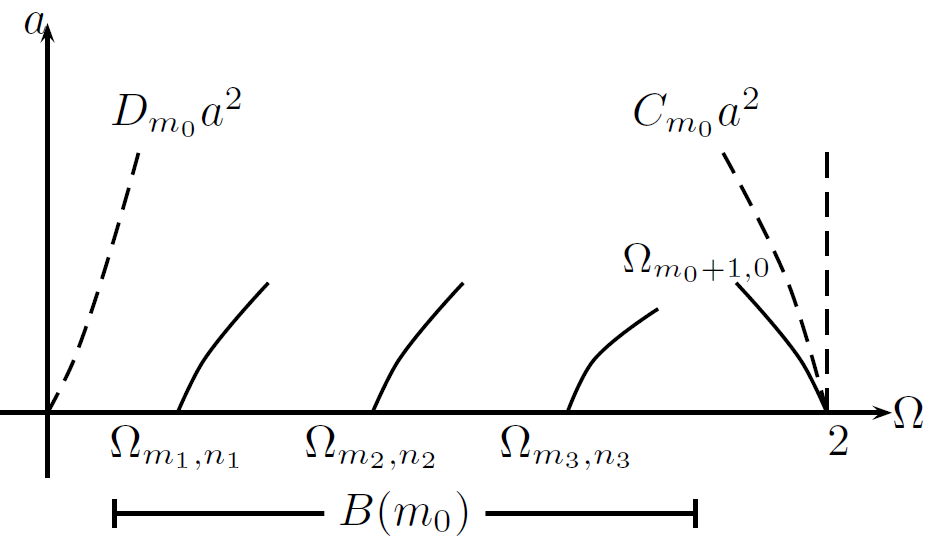}}
\end{center}

\caption{A schematic illustration of the bifurcation
curves in the parameter plane $(\Omega,a)$, where $a$ defines $\omega$.
The bifurcating solutions form
surfaces parameterized by $(\Omega,a)$ close to the curves $\Omega_{m,n}$.}%
\label{fig-scheme}%
\end{figure}

\begin{theorem}
\label{theorem-main} Fix an integer $m_{0}\in\mathbb{N}$ and denote
$\omega:=\mu+m_{0}\Omega$.

\begin{itemize}
\item[(i)] There exists a smooth family of radially symmetric vortices of
charge $m_{0}$ with a positive profile $U$ satisfying (\ref{statGP}) with
$\omega = \omega(a)$ given by
\[
\omega(a)=2(m_{0}+1) + \frac{(2m_{0})!}{4^{m_{0}} (m_{0}!)^{2}} a^{2}%
+\mathcal{O}(a^{4}),
\]
where the ``amplitude" $a$ parameterizes the family.

\item[(ii)] For $\Omega=0$ and small $a$, the vortices are degenerate saddle
points of the energy $E$ in (\ref{Energy}) with $2N(m_{0})$ negative
eigenvalues, a simple zero eigenvalue, and $2Z(m_{0})$ small eigenvalues of
order $\mathcal{O}(a^{2})$, where
\[
N(m_{0})=\frac{1}{2}m_{0}(m_{0}+1)\quad\mbox{\rm and}\quad Z(m_{0})=m_{0}.
\]

\item[(iii)] There exist $C_{m_{0}}>0$ and $D_{m_{0}}\geq0$ such that for
small $a$, $1+B(m_{0})$ global bifurcations occur when the parameter $\Omega$
is increased in the interval $[a^{2}D_{m_{0}},2-a^{2}C_{m_{0}}]$, where
\[
B(m_{0})=\frac{1}{2}m_{0}(m_{0}-1).
\]
For $\Omega\gtrsim a^{2}D_{m_{0}}$, the family of radially symmetric vortices
has only $2N(m_{0})$ negative eigenvalues and a simple zero eigenvalue, and it
losses two of these negative eigenvalues past each non-resonant bifurcation
point. If
$1\leq m_{0}\leq16$, the family has $2(m_{0}-1)$ negative eigenvalues for
$\Omega\gtrsim2-a^{2}C_{m_{0}}$.

\item[(iv)] A new smooth family of multi-vortex configurations  is connected
to the family of radially symmetric vortices on one side of each non-resonant
bifurcation point (of the pitchfork type). Furthermore, on the right
(respectively, left) side of the bifurcation point, the new family has one
more (respectively, one less) negative eigenvalue compared to the family of
radially symmetric vortices.

\item[(v)] For a non-resonant bifurcation point $\Omega_{m,n}\in(0,2)$ with $m > m_0$ and $n \geq 0$, the
new family has a polygon configuration of $(m-m_{0})$ charge-one vortices
surrounding a center with total charge $2m_{0}-m$.
For the \textquotedblleft last" bifurcation point $\Omega_{m_{0}%
+1,0}=2+\mathcal{O}(a^{2})$, the new family consists of the
charge-one asymmetric vortex ($m_0 = 1$), the asymmetric pair of charge-one
vortices ($m_0 = 2$), and a configuration of vortices near the center
of total charge $m_0$ ($m_0 \geq 3$).
\end{itemize}
\end{theorem}

\begin{remark}
By global bifurcation, we mean that the bifurcating family that originates
from the family of radially symmetric vortices of charge $m_{0}$ either
reaches the boundaries $\Omega=0$ or $\Omega=2$, diverges to infinity for a
value of $\Omega\in(0,2)$, or returns to another bifurcation point along the
family of radially symmetric vortices of charge $m_{0}$.
\end{remark}

\begin{remark}
For $1\leq m_{0}\leq3$, we have $D_{m_{0}}=0$ in item (iii), therefore, the
$1+B(m_{0})$ global bifurcations arise when $\Omega$ is increased from
$\Omega=0$ to $\Omega=2-a^{2}C_{m_{0}}$. However, we do not know if $D_{m_{0}}=0$ 
in a general case. If $D_{m_{0}}\neq0$, up to $Z(m_{0})$ additional
bifurcations may appear if $\Omega$ is increased from $\Omega=0$ to
$\Omega=a^{2}D_{m_{0}}$.
\end{remark}

\begin{remark}
For $1\leq m_{0}\leq3$, all bifurcation points are non-resonant in items
(iii)--(v). Resonant bifurcation points may exist in a general case for
$m_{0}\geq4$. In this case, the statements (i)-(iii) remain valid, but for
each bifurcation point of multiplicity $k$, the family of radially symmetric
vortices losses $2k$ negative eigenvalues past the bifurcation point. In the resonant case,
the statements (iv)--(v) require further estimates. However, these resonances are
unlikely to be present as the more likely scenario is that the multiple
eigenvalues at $a=0$ split into simple nonzero eigenvalues of order $\mathcal{O}(a^{2})$.
\end{remark}

\begin{remark}
For $m_{0} \geq4$, there are $R(m_{0})$ additional bifurcations near
$\Omega_{0} = 2$. For $4\leq m_{0}\leq16$, the additional bifurcation arise
past the last bifurcation point at $\Omega_{m_0+1,0}$. For
$m_{0}\geq17$, some of the $R(m_{0})$ bifurcations arise before the ``last"
bifurcation point. We have found numerically that $R(4) = R(5) = 1$, $R(6) =
2$, $R(7) = R(8) = 3$, etc.
\end{remark}

From a technical point of view, the proof of Theorem \ref{theorem-main} is developed by
using the equivariance of the bifurcation problem under the action of the
group $O(2)\times O(2)$. The global bifurcation result is proven by using the restriction
of the bifurcation problem to the fixed-point space of a dihedral group. This restriction
leads to a simple eigenvalue in the fixed-point space, which allows us to apply the global
Crandall--Rabinowitz result, see Theorem 3.4.1 of \cite{Ni2001}. This method is also helpful
to get additional information on the symmetries of the bifurcating solutions which is
essential to localize the distributions of zeros for the individual vortices in the
multi-vortex configurations.

The paper is structured as follows. In Section 2, we review eigenvalues of the
Schr\"{o}dinger operator for quantum harmonic oscillator and give definitions
for the primary and secondary branches of multi-vortex solutions. In Section
3, we analyze distribution of eigenvalues of the Hessian operators along the
primary branches at the secondary bifurcation points. In Section 4, we justify
bifurcations of the secondary branches at the non-resonant bifurcation points
by using bifurcation theorems. In Section 5, we study distribution of
individual vortices in the multi-vortex configurations along the secondary branches.

\section{Preliminaries}

We denote the space of square integrable functions on the plane by
$L^{2}(\mathbb{R}^{2})$ and the space of radially symmetric squared integrable
functions integrated with the weight $r$ by $L_{r}^{2}(\mathbb{R}^{+})$. We
also use the same notations for the $L^{2}$-based Sobolev spaces such as
$H^{2}(\mathbb{R}^{2})$ and $H_{r}^{2}(\mathbb{R}^{+})$. The
weighted subspace of $L^{2}$ with $\left\|  |\cdot|^{2} u \right\|  _{L^{2}} <
\infty$ are denoted by $L^{2,2}(\mathbb{R}^{2})$ and $L_{r}^{2,2}(\mathbb{R}^{+})$.

We distinguish notations for the two sets: $\mathbb{N} = \{1,2,3, \ldots\}$ and $\mathbb{N}_{0} = \{
0,1,2,3,\ldots\}$. Notation $b \lesssim a$ means that there is an $a$-independent constant $C$
such that $b \leq C a$ for all $a > 0$ sufficiently small.
If $X$ is a Banach space, notation $u = \mathcal{O}_{X}(a)$ means that $\| u \|_{X} \lesssim a$
for all $a > 0$ sufficiently small. Similarly, $\omega = \mathcal{O}(a)$ means that
$|\omega| \lesssim a$ for all $a > 0$ sufficiently small.

\subsection{Schr\"{o}dinger operator for quantum harmonic oscillator}

Recall the quantum harmonic oscillator with equal frequencies in the space of
two dimensions \cite{QM,Mo}. In polar coordinates on $\mathbb{R}^{2}$, the
energy levels of the quantum harmonic oscillator are given by eigenvalues of
the Schr\"{o}dinger operator $L$ written as
\begin{equation}
L:=-\Delta_{(r,\theta)}+r^{2}:\;\;H^{2}(\mathbb{R}^{2})\cap L^{2,2}%
(\mathbb{R}^{2})\rightarrow L^{2}(\mathbb{R}^{2}), \label{Schrodinger}%
\end{equation}
where $\Delta_{(r,\theta)}=\partial_{r}^{2}+r^{-1}\partial_{r}+r^{-2}%
\partial_{\theta}^{2}$. As is well-known \cite{QM,Mo}, the eigenvalues of $L$
are distributed equidistantly and can be enumerated by two indices
$m\in\mathbb{Z}$ for the angular dependence and $n\in\mathbb{N}_{0}$ for the
number of zeros of the eigenfunctions in the radial direction. To be more
precise, the eigenfunction $f_{m,n}$ for the eigenvalue $\lambda_{m,n}$ can be
written in the form
\[
f_{m,n}(r,\theta)=e_{m,n}(r)e^{im\theta},\quad m\in\mathbb{Z},\quad
n\in\mathbb{N}_{0},
\]
where $e_{m,n}$ is an $L_{r}^{2}(\mathbb{R}^{+})$-normalized solution of the
differential equation
\begin{equation}
\left(  -\Delta_{m}+r^{2}\right)  e_{m,n}(r)=\lambda_{m,n}e_{m,n}%
(r),\quad\Delta_{m}:=\partial_{r}^{2}+r^{-1}\partial_{r}-r^{-2}m^{2}
\label{def.vmn}%
\end{equation}
with $n$ zeros on $\mathbb{R}^{+}$ and the eigenvalue $\lambda_{m,n}$ is given
explicitly as
\begin{equation}
\lambda_{m,n}=2(\left\vert m\right\vert +2n+1),\quad m\in\mathbb{Z},\quad
n\in\mathbb{N}_{0}. \label{eig-Schr}%
\end{equation}
In particular, $\lambda_{0,0}=2$ is simple, $\lambda_{1,0}=\lambda_{-1,0}=4$
is double, $\lambda_{2,0}=\lambda_{-2,0}=\lambda_{0,1}=6$ is triple, and so
on. For fixed $m\in\mathbb{Z}$, the spacing between the eigenvalues is $4$.
Multiplicity of an eigenvalue $\lambda=2\ell$ for $\ell\in\mathbb{N}$ is
$\ell$.

\subsection{Primary branches of radially symmetric vortices}

Stationary solutions of the GP equation (\ref{Ec}) are given in the form
$u(t,x,y)=e^{-i\mu t}U(x,y)$, where $U$ satisfies (\ref{statGP}) and
$\mu\in\mathbb{R}$ is a free parameter
which has the physical meaning of the chemical potential. In polar coordinates
$(r,\theta)$, $U$ satisfies the stationary GP equation in the form
\begin{equation}
\mu U=-\Delta_{(r,\theta)}U+r^{2}U+|U|^{2}U+i\Omega\partial_{\theta}U.
\label{EcStat}%
\end{equation}
Radially symmetric vortices of a fixed charge $m_{0}\in\mathbb{N}$ are given
in the form
\begin{equation}
U(r,\theta)=e^{im_{0}\theta}\psi_{m_{0}}(r),\quad\omega=\mu+m_{0}\Omega,
\label{vortex}%
\end{equation}
where $(\psi_{m_{0}},\omega)$ is a root of the nonlinear operator
\begin{equation}
f(u,\omega):\quad H_{r}^{2}(\mathbb{R}^{+})\cap L_{r}^{2,2}(\mathbb{R}%
^{+})\times\mathbb{R}\rightarrow L_{r}^{2}(\mathbb{R}^{+}),
\end{equation}
given by $f(u,\omega):=-\Delta_{m_{0}}u+r^{2}u+u^{3} -\omega u$.

By Theorem 1 in \cite{CoGa15}, for every $m_{0} \in\mathbb{N}$, there exists a
unique smooth family of radially symmetric vortices of charge $m_{0}$
parameterized locally by amplitude $a$ such that
\begin{equation}
\psi_{m_{0}}(r) \equiv\psi_{m_{0}}(r;a) = a e_{m_{0},0}(r) + \mathcal{O}%
_{H^{1}_{r}}(a^{3}) \label{ua}%
\end{equation}
and
\begin{equation}
\label{branch-1}\omega\equiv\omega_{m_{0}}(a) = \lambda_{m_{0},0} + a^{2}
\omega_{m_{0},0} + \mathcal{O}(a^{4}),
\end{equation}
where $\lambda_{m_{0},0} = 2(m_{0}+1)$, $\omega_{m_{0},0} = \| e_{m_{0},0}
\|_{L^{4}_{r}}^{4}$, and the normalization $\| e_{m_{0},0} \|_{L^{2}_{r}} = 1$
has been used. By using the explicit expression for the $L^{2}_{r}%
(\mathbb{R}^{+})$-normalized Hermite--Gauss solutions of the Schr\"{o}dinger
equation (\ref{def.vmn}) with $\lambda_{m_{0},0} = 2(m_{0}+1)$ given by
\begin{equation}
\label{HGfunctions}e_{m_{0},0}(r) = \frac{\sqrt{2}}{\sqrt{m_{0}!}} r^{m_{0}}
e^{-\frac{r^{2}}{2}}, \quad m_{0} \in\mathbb{N}_{0},
\end{equation}
we compute explicitly
\begin{equation}
\label{HGvalues}\omega_{m_{0},0} = \| e_{m_{0},0}\|_{L^{4}_{r}}^{4} =
\frac{(2m_{0})!}{4^{m_{0}} (m_{0}!)^{2}}.
\end{equation}

Since $e_{m_{0},0}(r) > 0$ for all $r > 0$, the property $\psi_{m_{0}}(r;a) >
0$, $r > 0$ holds\footnote{More general
vortex families with $n_{0}$ zeros on $\mathbb{R}^{+}$ have also been
constructed in \cite{CoGa15}, but our work will focus on the case $n_{0} =
0$.} at least for sufficiently small $a$. The family of radially symmetric vortices approximated by (\ref{ua}) and
(\ref{branch-1}) in the small-amplitude limit is referred to as \emph{the
primary branch}.

\begin{remark}
Item (i) in Theorem \ref{theorem-main} is just a reformulation of the result
of Theorem 1 in \cite{CoGa15}.
\end{remark}

Every solution $U$ of the stationary GP equation (\ref{EcStat})
is a critical point of the energy functional
\begin{equation}
E_{\mu}(u)=E(u)-\mu Q(u), \label{Lyapunov}%
\end{equation}
where $E(u)$ is given by (\ref{Energy}) and $Q(u)=\Vert u\Vert_{L^{2}}^{2}$.
Expanding $E_{\mu}(u)$ near the critical point $U$ given by (\ref{vortex})
with $u=U+v$, where $v$ is a perturbation term in $H^{1}(\mathbb{R}^{2})\cap
L^{2,1}(\mathbb{R}^{2})$, we obtain the quadratic form at the leading order
\[
E_{\mu}(U+v)-E_{\mu}(U)=\langle\mathcal{H}\mathbf{v},\mathbf{v}\rangle_{L^{2}%
}+\mathcal{O}(\Vert\mathbf{v}\Vert_{H^{1}\cap L^{2,1}}^{3}),
\]
where the bold notation $\mathbf{v}$ is used for an augmented vector with
components $v$ and $\bar{v}$ and the Hessian operator $\mathcal{H}$ can be
defined in the stronger sense as the linear operator
\begin{equation}
\mathcal{H}:\quad H^{2}(\mathbb{R}^{2})\cap L^{2,2}(\mathbb{R}^{2})\rightarrow
L^{2}(\mathbb{R}^{2}), \label{Hessian-definition}%
\end{equation}
with
\begin{equation}
\mathcal{H}=\left[
\begin{array}
[c]{cc}%
-\Delta_{(r,\theta)}+r^{2}+i\Omega\partial_{\theta}-\mu+2\psi_{m_{0}}^{2} &
\psi_{m_{0}}^{2}e^{2im_{0}\theta}\\
\psi_{m_{0}}^{2}e^{-2im_{0}\theta} & -\Delta_{(r,\theta)}+r^{2}-i\Omega
\partial_{\theta}-\mu+2\psi_{m_{0}}^{2}%
\end{array}
\right]  . \label{Hessian}%
\end{equation}

By using the Fourier series
\begin{equation}
\label{Fourier-series}v = \sum_{m \in\mathbb{Z}} V_{m} e^{i m \theta},
\quad\bar{v} = \sum_{m \in\mathbb{Z}} W_{m} e^{i m \theta},
\end{equation}
the operator $\mathcal{H}$ is block diagonalized into blocks $H_{m}$ that acts
on $V_{m}$ and $W_{m-2m_{0}}$ for $m \in\mathbb{Z}$. We recall that
$\psi_{m_{0}} = \psi_{m_{0}}(\cdot;a)$, $\omega= \mu+ m_{0} \Omega=
\omega_{m_{0}}(a)$, and write the blocks $H_{m}$ as linear operators
\begin{equation}
\label{Hessian-blocks-definition}H_{m} : \quad H^{2}_{r}(\mathbb{R}^{+}) \cap
L^{2,2}_{r}(\mathbb{R}^{+}) \to L^{2}_{r}(\mathbb{R}^{+}),
\end{equation}
with explicit dependence on the parameters $(a,\Omega)$ as follows:
\begin{equation}
\label{Hessian-blocks}H_{m}(a,\Omega) = K_{m}(a) - \Omega(m-m_{0}) R,
\end{equation}
where
\begin{equation*}
\label{operator-K}K_{m}(a) = \left[
\begin{array}
[c]{cc}%
- \Delta_{m} + r^{2} - \omega_{m_{0}}(a) + 2 \psi^{2}_{m_{0}}(r;a) &
\psi_{m_{0}}^{2}(r;a)\\
\psi_{m_{0}}^{2}(r;a) & - \Delta_{m-2m_{0}} + r^{2} - \omega_{m_{0}}(a) + 2
\psi^{2}_{m_{0}}(r;a)
\end{array}
\right]
\end{equation*}
and
\[
R = \left[
\begin{array}
[c]{cc}%
1 & 0\\
0 & -1
\end{array}
\right]  .
\]

\emph{A secondary bifurcation} along the primary branch of radially symmetric
vortices given by (\ref{ua}) corresponds to the nonzero solutions in
$H^{2}_{r}(\mathbb{R}^{+}) \cap L^{2,2}_r(\mathbb{R}^+)$ of the spectral problem
\begin{equation}
\label{bif-prob}K_{m}(a) \left[
\begin{array}
[c]{l}%
V_{m}\\
W_{m-2m_{0}}%
\end{array}
\right]  = \Omega(m-m_{0}) R \left[
\begin{array}
[c]{l}%
V_{m}\\
W_{m-2m_{0}}%
\end{array}
\right]  .
\end{equation}
This spectral problem (\ref{bif-prob}) coincides with the stability problem for the primary
branch (\ref{ua}) in the absence of rotation. The spectral parameter $\lambda$
of the stability problem\footnote{When the vortex is unstable, a complex
eigenvalue $\lambda$ of the stability problem does not correspond to the
secondary bifurcation associated with the eigenvalue problem (\ref{bif-prob}%
).} is given for each $m \in\mathbb{Z}$ by $\lambda:= \Omega(m-m_{0})$. The
parameter $m$ for the angular mode satisfying the eigenvalue problem
(\ref{bif-prob}) corresponds to the bifurcating mode superposed on the primary
branch of vortex solutions.

\subsection{Secondary branches of multi-vortex solutions}

We can look for the secondary branches bifurcating along the primary branch of
radially symmetric vortices given by (\ref{ua}) and (\ref{branch-1}).
Consequently, we write
\begin{equation}
U(r,\theta) = e^{i m_{0} \theta} \psi_{m_{0}}(r;a) + v(r,\theta), \label{bif}%
\end{equation}
where $v$ is a root of the nonlinear operator
\begin{equation}
\label{g-function}g(v;a,\Omega) : \quad H^{2}(\mathbb{R}^{2}) \cap
L^{2,2}(\mathbb{R}^{2}) \times\mathbb{R} \times\mathbb{R} \to L^{2}%
(\mathbb{R}^{2}),
\end{equation}
given by
\begin{align}
g(v;a,\Omega)  &  = - \Delta_{(r,\theta)} v + r^{2} v + i
\Omega\left(  \partial_{\theta} v - i m_{0} v \right)  + 2 \psi_{m_{0}}%
^{2}(r;a) v + e^{2i m_{0} \theta} \psi_{m_{0}}^{2}(r;a) \bar{v}\nonumber\\
&  \phantom{t} + e^{-i m_{0} \theta} \psi_{m_{0}}(r;a) v^{2} + 2 e^{i m_{0}
\theta} \psi_{m_{0}}(r;a) |v|^{2} + |v|^{2} v -\omega_{m_{0}}(a) v. \label{operator-g}%
\end{align}
The Jacobian operator of $g(v;a,\Omega)$ at $v = 0$ is given by the Hessian
operator (\ref{Hessian}), which is block-diagonalized by the Fourier series
(\ref{Fourier-series}) into blocks (\ref{Hessian-blocks-definition})--(\ref{Hessian-blocks}).

In the next two lemmas, we analyze symmetries of the individual blocks of the
spectral problem (\ref{bif-prob}).

\begin{lemma}
\label{lemma-1} There exists $a_{0}$ such that for every $0 < a < a_{0}$, the
spectrum of $H_{m_{0}}(a,\Omega) = K_{m_{0}}(a)$ is strictly positive except
for a simple zero eigenvalue, which is related to the gauge symmetry spanned
by the eigenvector
\begin{equation}
\label{kernel}K_{m_{0}}(a)\left[
\begin{array}
[c]{l}%
\psi_{m_{0}}\\
-\psi_{m_{0}}%
\end{array}
\right]  = \left[
\begin{array}
[c]{l}%
0\\
0
\end{array}
\right]  .
\end{equation}
Consequently, no bifurcations arise in the $\Omega$ continuation from the
block $H_{m_{0}}(a,\Omega) = K_{m_{0}}(a)$.
\end{lemma}

\begin{proof}
For $m = m_{0}$, $H_{m_{0}}(a,\Omega) = K_{m_{0}}(a)$ is independent of the
rotation frequency $\Omega$. If the primary branch (\ref{ua}) describes
vortices with $\psi_{m_{0}}(r;a) > 0$ for all $r > 0$, that is, if $n_{0} =
0$, then the assertion on the spectrum of $K_{m_{0}}(a)$ for small $a$ follows
from the previous works \cite{ChPel,Kollar}.
\end{proof}

\begin{lemma}
\label{lemma-2} Eigenvalues of the spectral problem (\ref{bif-prob}) with $m <
m_{0}$ are identical to eigenvalues of the spectral problem (\ref{bif-prob})
for $m > m_{0}$.
\end{lemma}

\begin{proof}
We observe the symmetry $\Delta_{m} = \Delta_{m_{0} + (m - m_{0})}$ and
$\Delta_{m-2m_{0}} = \Delta_{m_{0} - (m-m_{0})}$ with respect to the symmetry
point at $m = m_{0}$. As a result, for each $k \in\mathbb{N}$, if $\lambda=
\Omega k$ is an eigenvalue of the spectral problem (\ref{bif-prob}) with $m =
m_{0} + k$ for the eigenvector $[V_{m_{0}+k},W_{-m_{0}+k}]$, then $\lambda=
\Omega k$ is the same eigenvalue of the spectral problem (\ref{bif-prob}) with
$m = m_{0} - k$ for the eigenvector $[V_{m_{0}-k},W_{-m_{0}-k}] =
[W_{-m_{0}+k},V_{m_{0}+k}]$.
\end{proof}

\vspace{0.25cm}

It follows from Lemmas \ref{lemma-1} and \ref{lemma-2} that it is sufficient
to consider the spectrum of $H_{m}(a,\Omega)$ for $m > m_{0}$ and to count
negative and zero eigenvalues of $H_{m}(a,\Omega)$ in pairs. If $a=0$ and
$\Omega=0$, we have $H_{m}(0,0)=K_{m}(0)$, where
\begin{equation}
K_{m}(0)=\left[
\begin{array}
[c]{cc}%
-\Delta_{m}+r^{2}-\lambda_{m_{0},0} & 0\\
0 & -\Delta_{m-2m_{0}}+r^{2}-\lambda_{m_{0},0}%
\end{array}
\right]  . \label{bifurcations-K}%
\end{equation}
The spectrum of $K_{m}(0)$ is obtained from eigenvalues of the Schr\"{o}dinger
equation (\ref{def.vmn}). The first diagonal entry of
$K_{m}(0)$ has strictly positive eigenvalues
\[
\mu_{m,n}^{+}(0):=2(m+2n-m_{0})>0,\quad m>m_{0},\quad n\in\mathbb{N}_{0}.
\]
The second diagonal entry of $K_{m}(0)$ has eigenvalues
\[
\mu_{m,n}^{-}(0):=2(|m-2m_{0}|+2n-m_{0}),\quad m>m_{0},\quad n\in
\mathbb{N}_{0}.
\]
Let $N(m_{0})$ and $Z(m_{0})$ be the cardinality of the sets
\[
\mathcal{N}(m_{0})=\left\{  m>m_{0},\quad n\in\mathbb{N}_{0}:\quad\mu
_{m,n}^{-}(0)<0\right\}
\]
and
\[
\mathcal{Z}(m_{0})=\left\{  m>m_{0},\quad n\in\mathbb{N}_{0}:\quad\mu
_{m,n}^{-}(0)=0\right\}  .
\]
The following lemma gives the count of $N(m_{0})$ and $Z(m_{0})$.

\begin{lemma}
\label{lemma-count-N-Z} For every $m_{0} \in\mathbb{N}$, we have
\begin{equation}
\label{count-N-Z}N(m_{0}) = \frac{m_{0}(m_{0}+1)}{2}, \quad Z(m_{0}) = m_{0}.
\end{equation}

\end{lemma}

\begin{proof}
To count $Z(m_{0})$, we note that $\mu_{m,n}^{-}(0)=0$ if and only if
$|m-2m_{0}|+2n=m_{0}$. The cardinality of the set $\{(\ell,n)\in
\mathbb{Z}\times\mathbb{N}_{0}:\quad|\ell|+2n=m_{0}\}$ coincides with the
multiplicity of the eigenvalue $\lambda_{m_{0},0}=2(m_{0}+1)$ of the
Schr\"{o}dinger equation (\ref{def.vmn}), which is $m_{0}+1$. Since
$|\ell|\leq m_{0}$ translates to $m_{0}\leq m\leq3m_{0}$ and since $m=m_{0}$
contains one zero eigenvalue with $n=0$, we obtain $Z(m_{0})=m_{0}+1-1=m_{0}$.

To count $N(m_{0})$, we follow the same idea. The largest negative eigenvalue
$\mu_{m,n}^{-}(0)=-2$ corresponds to $|m-2m_{0}|+2n=m_{0}-1$, which coincides
with the multiplicity of the eigenvalue $\lambda_{m_{0}-1,0}=2m_{0}$, which is
$m_{0}$. The next negative eigenvalue $\mu_{m,n}^{-}(0)=-4$ corresponds to
$|m-2m_{0}|+2n=m_{0}-2$, which coincides with the multiplicity of the
eigenvalue $\lambda_{m_{0}-2,0}=2(m_{0}-1)$, which is $m_{0}-1$. The count continues until
we reach the smallest negative eigenvalue $\mu_{m,n}^{-}(0)=-2m_{0}$, which
corresponds to $|m-2m_{0}|+2n=0$ and which is simple for $m=2m_{0}$ and $n=0$.
Summing integers from $1$ to $m_{0}$, we obtain $N(m_{0})=1+2+\cdots
+m_{0}=m_{0}(m_{0}+1)/2$.
\end{proof}

\begin{remark}
Lemma \ref{lemma-count-N-Z} yields the proof of item (ii) of Theorem
\ref{theorem-main}.
\end{remark}

Let us give some explicit examples. If $m_{0} = 1$, then $\lambda_{1,0} = 4$
and
\begin{align}
\label{bifurcation-1}\left\{
\begin{array}
[c]{l}%
\sigma(K_{2}) = \{ -2, 2, 2, 6, 6, \cdots\},\\
\sigma(K_{3}) = \{ 0, 4, 4, 8, 8, \cdots\},\\
\sigma(K_{4}) = \{ 2, 6, 6, 10, 10, \cdots\},\\
\vdots
\end{array}
\right.
\end{align}
so that $N(1) = 1$ and $Z(1) = 1$.

If $m_{0} = 2$, then $\lambda_{2,0} = 6$ and
\begin{align}
\label{bifurcation-2}\left\{
\begin{array}
[c]{l}%
\sigma(K_{3}) = \{ -2, 2, 2, 6, 6, \cdots\},\\
\sigma(K_{4}) = \{ -4, 0, 4, 4, 8, 8, \cdots\},\\
\sigma(K_{5}) = \{ -2, 2, 6, 6, 10, 10, \cdots\},\\
\sigma(K_{6}) = \{ 0, 4, 8, 8, 12, 12, \cdots\},\\
\sigma(K_{7}) = \{ 2, 6, 10, 10, 14, 14, \cdots\},\\
\vdots
\end{array}
\right.
\end{align}
so that $N(2) = 3$ and $Z(2) = 2$.

If $m_{0} = 3$, then $\lambda_{3,0} = 8$ and
\begin{align}
\label{bifurcation-3}\left\{
\begin{array}
[c]{l}%
\sigma(K_{4}) = \{ -2, 2, 2, 6, 6, \cdots\},\\
\sigma(K_{5}) = \{ -4, 0, 4, 4, 8, 8, \cdots\},\\
\sigma(K_{6}) = \{ -6, -2, 2, 6, 6, 10, 10, \cdots\},\\
\sigma(K_{7}) = \{ -4, 0, 4, 8, 8, 12, 12, \cdots\},\\
\sigma(K_{8}) = \{ -2, 2, 6, 10, 10, 14, 14, \cdots\},\\
\sigma(K_{9}) = \{ 0, 4, 8, 12, 12, 16, 16, \cdots\},\\
\sigma(K_{10}) = \{ 2, 6, 10, 14, 14, 18, 18, \cdots\},\\
\vdots
\end{array}
\right.
\end{align}
so that $N(3) = 6$ and $Z(3) = 3$.

In what follows, we fix $a > 0$ small enough and consider a continuation of
eigenvalues of $H_{m}(a,\Omega)$ given by (\ref{Hessian-blocks}) with respect
to the parameter $\Omega$ in the interval $(0,2)$. When one of the eigenvalues
of $H_{m}(a,\Omega)$ reaches zero, we say that a secondary bifurcation occurs
along the primary branch of radially symmetric vortices given by (\ref{ua})
and (\ref{branch-1}).

We will show that for every $m=m_{0}+2\ell$, $1\leq\ell\leq m_{0}$, there is
an $a$-independent constant $D_{m,m_{0}}\geq0$ such that the zero eigenvalue
of $K_{m}(0)$ becomes a positive eigenvalue of $H_{m}(a,\Omega)$ for small $a$
and for $\Omega\gtrsim D_{m,m_{0}}a^{2}$. The maximum of $D_{m_{0}+2\ell
,m_{0}}$ for $1\leq\ell\leq m_{0}$ is denoted by $D_{m_{0}}$.

We further show that there is another $a$-independent constant $C_{m_{0}} > 0$
such that when $\Omega$ is increased in the interval $(D_{m_{0}}%
a^{2},2-C_{m_{0}}a^{2})$, then $1+B(m_{0})$ secondary bifurcations occur,
where $B(m_{0}) = m_0 (m_0-1)/2$, at which a negative eigenvalue of $H_{m}%
(a,\Omega)$ for some $m$ and for $\Omega$ below the bifurcation point becomes
a positive eigenvalue of $H_{m}(a,\Omega)$ for the same $m$ and for $\Omega$
above the bifurcation point. The first $B(m_{0})$ secondary bifurcations occur
for values of $\Omega$ sufficiently distant from the value $\Omega_{0} = 2$,
whereas the last secondary bifurcation occurs for the value of $\Omega$ near
but below the value $\Omega_{0} = 2$. The latter case has to be handled in the
presence of infinitely many resonances in the limit $a \to 0$. The aforementioned claims proved in
Section 3 will provide proofs of item (iii) in Theorem \ref{theorem-main}.

At each non-resonant bifurcation point, a new secondary branch of vortex
solutions is born for $\Omega$ on one side of the bifurcation point among the
roots of the nonlinear operator $g$ given by (\ref{g-function}) and (\ref{operator-g}). The secondary
branch represents a multi-vortex configuration near the origin of the total charge $m_{0}$,
where the radial symmetry is now broken. The aforementioned claims proved in
Section 4 and 5 will provide respectively proofs of items (iv) and (v) in
Theorem \ref{theorem-main}.

\section{Secondary bifurcations as $\Omega$ increases}

Let the primary branch of radially symmetric vortices be defined by (\ref{ua})
and (\ref{branch-1}) in the small-amplitude limit. Expanding the family of
operators $K_{m}(a)$ in powers of $a$, we obtain
\begin{align*}
K_{m}(a)  &  = \left[
\begin{array}
[c]{cc}%
- \Delta_{m} + r^{2} - \lambda_{m_{0},0} & 0\\
0 & - \Delta_{m-2 m_{0}} + r^{2} - \lambda_{m_{0},0}%
\end{array}
\right] \\
&  \phantom{t} + a^{2} \left[
\begin{array}
[c]{cc}%
-\omega_{m_{0},0} + 2 e_{m_{0},0}^{2}(r) & e_{m_{0},0}^{2}(r)\\
e_{m_{0},0}^{2}(r) & -\omega_{m_{0},0} + 2e_{m_{0},0}^{2}(r)
\end{array}
\right]  + \mathcal{O}(a^{4}),
\end{align*}
where the correction term is given by a bounded potential on $\mathbb{R}^{+}$.

Also recall from (\ref{Hessian-blocks}) that the operator $H_{m}(a,\Omega)$ is
expanded as $a \to0$ with the leading-order term given by the diagonal
operator $H_{m}(0,\Omega)$ with the entries given by two linear operators:
\begin{equation}
\left\{
\begin{array}
[c]{l}%
L_{+} := -\Delta_{m}+r^{2}-\lambda_{m_{0},0}-\Omega(m-m_{0}),\\
L_{-} := -\Delta_{m-2m_{0}}+r^{2}-\lambda_{m_{0},0}+\Omega(m-m_{0}).
\end{array}
\right.  \label{operator-leading-order}%
\end{equation}

We shall now analyze how eigenvalues of $H_{m}(a,\Omega)$ cross zero when
$\Omega$ is increased in the interval $(0,2)$.

\subsection{Zero eigenvalues of $K_{m}(0)$}

When $a = 0$ and $\Omega= 0$, each operator block $H_{m}(0,0) = K_{m}(0)$ has
a simple zero eigenvalue for $m = m_{0}+2 \ell$, $1 \leq\ell\leq m_{0}$. See
examples in (\ref{bifurcation-1}), (\ref{bifurcation-2}), and
(\ref{bifurcation-3}). The following lemma tells us that the zero eigenvalue
of such $H_{m}(0,0)$ becomes a positive eigenvalue of $H_{m}(a,\Omega)$ for
every sufficiently small $a$, provided the values of $\Omega$ are sufficiently
large and positive.

\begin{lemma}
\label{lemma-zero-eig} For every $m_{0} \in\mathbb{N}$, there exists $a_{0} >
0$ and $D_{m_{0}} \geq0$ such that for every $0 < a < a_{0}$, $\Omega>
D_{m_{0}} a^{2}$, and $1 \leq\ell\leq m_{0}$, there is a small positive
eigenvalue of $H_{m_{0}+2\ell}(a,\Omega)$ which is continuous in $(a,\Omega)$
and converges to the zero eigenvalue of $K_{m_{0}+2\ell}(0)$ as $a \to0$ and
$\Omega\to0$.
\end{lemma}

\begin{proof}
The zero eigenvalue of $K_{m}(0)$ for $m = m_{0} + 2 \ell$, $1 \leq\ell\leq
m_{0}$ corresponds to the second diagonal operator in $K_{m}(0)$. Let us show
by the perturbation theory argument that the zero eigenvalue is continued as a
small $\mathcal{O}(a^{2})$ eigenvalue of $K_{m}(a)$ for all $a$ sufficiently small.

The eigenfunction of $K_{m}(0)$ with $m = m_{0} + 2 \ell$ for the zero
eigenvalue is obtained from the balance
\[
\lambda_{m-2m_{0},n} = \lambda_{m_{0},0} \quad\Rightarrow\quad n(\ell) =
\frac{m_{0} - |2\ell- m_{0}|}{2}.
\]
Since the zero eigenvalue of $K_{m_{0} + 2 \ell}(0)$ is simple, the regular
perturbation theory in \cite{Kato} implies the existence of a small eigenvalue $\mu_{\ell
}(a)$ of the linear operator $K_{m_{0} + 2 \ell}(a)$ and the corresponding
eigenvector $(V_{m_{0} + 2 \ell},W_{-m_{0} + 2 \ell})$, which are analytic functions of $a$.
Their Taylor expansions are given by
\begin{align}
\label{1-parameter-perturbation}\left\{
\begin{array}
[c]{lcl}%
V_{m_{0} + 2 \ell} & = & a^{2} \tilde{V}_{m_{0} + 2 \ell} + \mathcal{O}_{L^2_r}%
(a^{4}),\\
W_{-m_{0} + 2 \ell} & = & c_{-m_{0} + 2 \ell} e_{|m_{0}-2\ell|,n(\ell)} +
a^{2} \tilde{W}_{-m_{0} + 2 \ell} + \mathcal{O}_{L^2_r}(a^{4}),\\
\mu_{\ell} & = & a^{2} \tilde{\mu}_{\ell} + \mathcal{O}(a^{4}),
\end{array}
\right.
\end{align}
where $c_{-m_0+2\ell} \neq 0$ is arbitrary, $ \tilde{V}_{m_{0} + 2 \ell}$,
$\tilde{W}_{-m_{0} + 2 \ell}$, and $\tilde{\mu}_{\ell}$ are obtained by
the standard projection algorithm, and the correction terms are defined uniquely
by the method of Lyapunov--Schmidt reductions. In particular,
$\tilde{\mu}_{\ell}$ is obtained from
\begin{equation}
\label{mu-perturbation}\tilde{\mu}_{\ell} = - \omega_{m_{0},0} + 2 \langle
e_{m_{0},0}^{2}, e_{|m_{0}-2\ell|,n(\ell)}^{2} \rangle_{L^2_r}.
\end{equation}
If $\tilde{\mu}_{\ell} \neq0$, the eigenvalue $\mu_{\ell}(a)$ is generally
nonzero but $\mathcal{O}(a^{2})$ small.

It follows from (\ref{operator-leading-order}) that the
$\Omega$-term in $L_{-}$ is a positive perturbation to $K_{m}(0)$ for $m > m_{0}$. Therefore,
there exists an $a$-independent constant $D_{\ell,m_{0}} \geq0$ such that the
eigenvalue $\mu_{\ell}(a)$ continued with respect to the parameter $\Omega$ is
strictly positive for $\Omega> D_{\ell,m_{0}} a^{2}$. The assertion of the lemma is proved by
taking the largest of
$D_{\ell,m_{0}}$ for all admissible $1 \leq\ell\leq m_{0}$ as $D_{m_{0}}$.
\end{proof}

\begin{remark}
Lemma \ref{lemma-zero-eig} yields the existence of constant $D_{m} \geq0$ in
item (iii) of Theorem \ref{theorem-main}.
\end{remark}

\begin{remark}
If $\tilde{\mu}_{\ell} > 0$ in the perturbation result (\ref{mu-perturbation}),
then $\mu_{\ell} > 0$ for every small $a > 0$ and $\Omega > 0$. If this is true
for every $1 \leq\ell\leq m_{0}$, then $D_{m_{0}} = 0$ in Lemma \ref{lemma-zero-eig}.
In particular, this is true for $1
\leq m_{0} \leq3$. Indeed, for $\ell= m_{0}$ and $n(\ell) = 0$, we obtain from
(\ref{HGvalues}) and (\ref{mu-perturbation}):
\[
\tilde{\mu}_{m_{0}} = \| e_{m_{0},0}\|_{L^{4}_{r}}^{4} > 0.
\]
For $\ell= m_{0} - 1$ and $n(\ell) = 1$, we use the following formula for the
$L^{2}_{r}(\mathbb{R}^{+})$-normalized Hermite--Gauss solutions of the
Schr\"{o}dinger equation (\ref{def.vmn}) with $\lambda_{m,1} = 2 (m+3)$:
\begin{equation}
\label{H-1-functions}e_{m,1}(r) = \frac{\sqrt{2}}{\sqrt{(m+1)!}} r^{m}
(m+1-r^{2}) e^{-\frac{r^{2}}{2}}.
\end{equation}
Then, we obtain from (\ref{mu-perturbation}) for $m_{0} \geq2$:
\[
\tilde{\mu}_{m_{0}-1} = 2 \langle e_{m_{0},0}^{2}, e_{m_{0}-2,1}^{2} \rangle_{L^2_r}
- \| e_{m_{0},0}\|_{L^{4}_{r}}^{4} = \frac{(2m_{0})! (m_{0}^{2}+m_{0}%
-1)}{4^{m_{0}} (m_{0}!)^{2}} > 0.
\]
By the symmetry, we also have $\mu_{1} = \mu_{m_{0}-1} > 0$. Thus, for $1 \leq
m_{0} \leq3$, we have $\tilde{\mu}_{\ell} > 0$ for all admissible $1 \leq
\ell\leq m_{0}$.
\end{remark}

\begin{remark}
It remains unclear if $\tilde{\mu}_{\ell} > 0$ for the other values in $2
\leq\ell\leq m_{0} - 2$ for $m_{0} \geq4$.
\end{remark}

\subsection{Zero eigenvalues of $H_{m}(a,\Omega)$ for $\Omega\in(0,2)$}

For $m>m_{0}$, the leading-order diagonal operator $H_{m}(0,\Omega)$ given by
the operators $L_{+}$ and $L_{-}$ in (\ref{operator-leading-order}) has an
eigenbasis
\begin{equation}
\{(e_{m,n},0);(0,e_{|m-2m_{0}|,n})\}_{n\in\mathbb{N}_{0}}.
\label{basis-bifurcation}%
\end{equation}
Since $H_{m}(a,\Omega)$ is self-adjoint, by regular perturbation theory in \cite{Kato}, the
operator $H_{m}(a,\Omega)$ has a set of eigenvalues counted by $n\in
\mathbb{N}_{0}$:
\begin{equation}
\left\{
\begin{array}
[c]{ll}%
\mu_{m,n}^{+}(a,\Omega) & := \lambda_{m,n}-\lambda_{m_{0},0}-\Omega
(m-m_{0})+\mathcal{O}(a^{2}),\\
\mu_{m,n}^{-}(a,\Omega) & :=\lambda_{m-2m_{0},n}-\lambda_{m_{0},0}%
+\Omega(m-m_{0})+\mathcal{O}(a^{2}).
\end{array}
\right.  \label{mu-plus-minus}%
\end{equation}
For $m>m_{0}$, $n\in\mathbb{N}_{0}$, and $\Omega<2$, we have
\[
\lambda_{m,n}-\lambda_{m_{0},0}-\Omega(m-m_{0})=(2-\Omega)(m-m_{0})+4n>0.
\]
Therefore, the eigenvalues $\mu_{m,n}^{+}(a,\Omega)$ never become zero for
small $a$ and $\Omega<2$. On the other hand, the eigenvalues $\mu_{m,n}%
^{-}(a,\Omega)$ become zero when $\Omega=\Omega_{m,n}(a)$ given by
\begin{equation}
\Omega_{m,n}(a)=2\frac{m_{0}-|m-2m_{0}|}{m-m_{0}}-\frac{4n}{m-m_{0}%
}+\mathcal{O}(a^{2})\text{.} \label{bifurcation-point}%
\end{equation}
Let $B(m_{0})$ denote the number of eigenvalues $\mu_{m,n}^{-}$ crossing zero
at $\Omega=\Omega_{m,n}(a)$ with $\Omega_{m,n}(0)\in(0,2)$. The following
lemma gives the count of $B(m_{0})$.

\begin{lemma}
\label{lemma-count-B} For every $m_{0} \in\mathbb{N}$, we have
\begin{equation}
\label{count-B}B(m_{0}) = \frac{m_{0}(m_{0}-1)}{2}.
\end{equation}

\end{lemma}

\begin{proof}
To count $B(m_{0})$, we count the values of $m > m_{0}$ for the first values
of $n \in\mathbb{N}_{0}$, when $\Omega_{m,n}(0) \in(0,2)$:

\begin{itemize}
\item For $n = 0$, the inequality $0 < m_{0} - |m-2m_{0}| < m-m_{0}$ is true
for $2m_{0}+1 \leq m \leq3m_{0}-1$;

\item For $n = 1$, the inequality $0 < m_{0} - |m-2m_{0}| -2 < m-m_{0}$ is
true for $m_{0}+3 \leq m \leq3m_{0}-3$;

\item For $n = 2$, the inequality $0 < m_{0} - |m-2m_{0}| -4 < m-m_{0}$ is
true for $m_{0}+5 \leq m \leq3m_{0}-5$.
\end{itemize}
\noindent For a general $n \in\mathbb{N}$, we have $\Omega_{m,n}(0) \in(0,2)$ for $m_{0}
+ 2n+1 \leq m \leq3m_{0} - 2n-1$ provided that the range for $m$ is nonempty.
Summing up all cases, we have
\[
B(m_{0}) = m_{0} - 1 + \sum_{n=1}^{\infty} [2 m_{0} - 4n - 1]_{+}
\]
where $[a]_{+}$ is $a$ when $a \geq0$ and $0$ if $a < 0$. The sum is finite as
$n$ terminates at the last entry for which $2m_{0} - 4n - 1 > 0$. If $m_{0}$
is odd, then the last entry corresponds to $N = (m_{0}-1)/2$ and we obtain
\[
\sum_{n=1}^{\infty} [2 m_{0} - 4n - 1]_{+} = \sum_{n=1}^{N} (2m_{0}-4n-1) =
\frac{m_{0}^{2}-3m_{0}+2}{2}.
\]
If $m_{0}$ is even, then the last entry corresponds to $N = m_{0}/2-1$ and we
obtain
\[
\sum_{n=1}^{\infty} [2 m_{0} - 4n - 1]_{+} = \sum_{n=1}^{N} (2m_{0}-4n-1) =
\frac{m_{0}^{2}-3m_{0}+2}{2}.
\]
Adding $m_0 - 1$ to this number, we obtain (\ref{count-B}) in both cases.
\end{proof}

\vspace{0.25cm}

In particular, we have $B(1) = 0$, $B(2) = 1$, $B(3) = 3$, and $B(4) = 6$. See
examples in (\ref{bifurcation-1}), (\ref{bifurcation-2}), and
(\ref{bifurcation-3}). Let us list the bifurcation values of $\Omega$ for
these examples:

\begin{itemize}
\item For $m_{0} = 1$, no bifurcations occur;

\item For $m_{0} = 2$, the only bifurcation occurs at $\Omega_{5,0}(0) = 2/3$;

\item For $m_{0} = 3$, three bifurcations occur at $\Omega_{7,0}(0) = 1$,
$\Omega_{8,0}(0) = 2/5$, and $\Omega_{6,1}(0) = 2/3$;

\item For $m_{0} = 4$, six bifurcations occur at $\Omega_{9,0}(0) = 6/5$,
$\Omega_{10,0}(0) = 2/3$, $\Omega_{11,0}(0) = 2/7$, $\Omega_{7,1}(0) = 2/3$,
$\Omega_{8,1}(0) = 1$, and $\Omega_{9,1}(0) = 2/5$;
\end{itemize}

\begin{remark}
\label{remark-resonant-curve} Lemma \ref{lemma-count-B} yields the number
$B(m_{0})$ in item (iii) of Theorem \ref{theorem-main}. Note that the
bifurcation points of $\Omega$ are simple for $1\leq m_{0}\leq3$. Multiple
bifurcation points exist in a general case for $m_{0}\geq4$, e.g.
$\Omega_{10,0}(0)=\Omega_{7,1}(0)=2/3$ for $m_{0}=4$.
\end{remark}

The following proposition summarizes properties of $H_{m}(a,\Omega)$ near each
bifurcation point. These properties are needed for the bifurcation analysis in
Section 4.

\begin{proposition}
For every $m_{0}\in\mathbb{N}$, let $\Omega_{\ast}(a)$ be one of the
bifurcation points defined by (\ref{bifurcation-point}). Assume it has
multiplicity $k$ and corresponds to $m_{1},\ldots,m_{k}>m_{0}$. There exists
$a_{0}>0$, $C_{m_{0}}>0$, and $E_{m_{0}}>0$ such that for every $0<a<a_{0}$,
$|\Omega-\Omega_{\ast}(a)|<C_{m_{0}}a^{2}$, and every $m>m_{0}$ such that
$m\notin\{m_{1},\ldots,m_{k}\}$, the operator $H_{m}(a,\Omega)$ is invertible
in $L_{r}^{2}(\mathbb{R}^{+})$ with the bound
\begin{equation}
\Vert H_{m}(a,\Omega)^{-1}\Vert_{L_{r}^{2}\rightarrow H_{r}^{2}\cap
L_{r}^{2,2}}\leq E_{m_{0}},\quad m>m_{0},\quad m\notin\{m_{1},\ldots,m_{k}\}.
\label{bound-on-inverse-general}%
\end{equation}
Moreover, the number of negative eigenvalues of $H_{m}(a,\Omega)$,
$m\notin\{m_{1},\ldots,m_{k}\}$ remains the same for every $\Omega$ in
$|\Omega-\Omega_{\ast}(a)|<C_{m_{0}}a^{2}$. On the other hand, the number of
negative eigenvalues for $H_{m}(a,\Omega)$, $m\in\{m_{1},\ldots,m_{k}\}$ is
reduced by one when $\Omega$ crosses $\Omega_{\ast}(a)$ in $|\Omega
-\Omega_{\ast}(a)|<C_{m_{0}}a^{2}$. \label{proposition-secondary-bifurcations}
\end{proposition}

\begin{proof}
First, we note that for each $m > m_{0}$, there may be at most one eigenvalue
of $H_{m}(a,\Omega)$ which becomes zero at $\Omega= \Omega_{*}(a)$. Bound
(\ref{bound-on-inverse-general}) follows from the fact that $H_{m}%
(a,\Omega_{*}(a))$ with $m \notin\{m_{1},\ldots,m_{k}\}$ has no
eigenvalues in the neighborhood of zero. On the other hand, each simple eigenvalue of $H_{m}(a,\Omega
_{*}(a))$ with $m \in\{m_{1},\ldots,m_{k}\}$ is continued in $\Omega$
according to the derivative
\begin{equation}
\label{derivative-H}\frac{\partial H_{m}}{\partial \Omega}(a,\Omega) = -(m-m_{0}) R.
\end{equation}
Let $(V_{m},W_{m-2m_{0}})$ be the corresponding eigenvector for the zero
eigenvalue of $H_{m}(a,\Omega_{*}(a))$. Since $m > m_{0}$, the eigenvalue is
positive for $\Omega\gtrsim\Omega_{*}(a)$ and negative for $\Omega
\lesssim\Omega_{*}(a)$ if $S_{m} < 0$, where
\begin{equation}
\label{krein-sign}S_{m} := \| V_{m} \|_{L^{2}_{r}}^{2} - \| W_{m-2m_{0}}
\|^{2}_{L^{2}_{r}}.
\end{equation}
Since $V_{m} \to0$ as $a \to0$, we have $S_{m} < 0$ for each $m \in
\{m_{1},\ldots,m_{k}\}$, provided $a$ is small enough.
\end{proof}

\begin{remark}
The quantity $S_{m}$ defined by (\ref{krein-sign}) is referred to as the Krein
quantity. The sign of $S_{m}$ gives the Krein signature of the neutrally
stable eigenvalues of the spectral stability problem associated with the
radially symmetric vortices in the case of no rotation \cite{Kollar}.
\end{remark}

\begin{definition}
If $k = 1$ in Proposition \ref{proposition-secondary-bifurcations}, we say
that the bifurcation point $\Omega_{*}(a)$ is non-resonant.
\label{remark-non-resonance}
\end{definition}

\subsection{Zero eigenvalues of $H_{m}(a,\Omega)$ for $\Omega$ near $2$}

Consider the rotation frequency $\Omega= 2 + \mathcal{O}(a^{2})$. According to
(\ref{bif-prob}) and (\ref{bifurcations-K}), see examples in
(\ref{bifurcation-1}), (\ref{bifurcation-2}), and (\ref{bifurcation-3}), there
are infinitely many resonances for $a = 0$.

We will show that if $\Omega$ is defined at a particular value denoted by
$\Omega_{m_{0}+1,0}(a) = 2 + \mathcal{O}(a^{2})$, for which the spectral
stability problem (\ref{bif-prob}) with $m = m_{0} + 1$ admits a nontrivial
solution, then the blocks $H_{m}(a,\Omega)$ of the Hessian operator for every
$m \geq m_{0} + 2$ are invertible in $L^{2}_{r}(\mathbb{R}^{+})$ near $\Omega=
\Omega_{m_{0}+1,0}(a)$ and the smallest eigenvalue of $H_{m}(a,\Omega
_{m_{0}+1,0}(a))$ is proportional to $\mathcal{O}(a^{2})$. At the same time,
the block $H_{m_{0}+1}(a,\Omega_{m_{0}+1,0}(a))$ has a simple zero eigenvalue
and a simple positive eigenvalue proportional to $\mathcal{O}(a^{2})$. We also
show for $1 \leq m_{0} \leq16$ that the blocks $H_{m}(a,\Omega_{m_{0}%
+1,0}(a))$ for $m_{0}+2 \leq m \leq2m_{0}$ have exactly one small negative
eigenvalue proportional to $\mathcal{O}(a^{2})$, whereas all other eigenvalues
are strictly positive.

The following lemma gives the precise location of $\Omega_{m_{0}+1,0}(a) = 2 +
\mathcal{O}(a^{2})$.

\begin{lemma}
There exists $a_{0} > 0$ such that for every $0 < a < a_{0}$, there exists
$\Omega_{m_{0}+1,0}(a) < 2$ given asymptotically by
\begin{equation}
\label{frequency-1}\Omega_{m_{0}+1,0}(a) := 2 -\frac{(2m_{0})!}{4^{m_{0}}
m_{0}! (m_{0}+1)!} a^{2} + \mathcal{O}(a^{4}),
\end{equation}
such that $H_{m_{0}+1}(a,\Omega_{m_{0}+1,0}(a))$ has a simple zero eigenvalue.
\label{lemma-last-bifurcation}
\end{lemma}

\begin{proof}
We solve the bifurcation equation (\ref{bif-prob}) for $m = m_{0} + 1$ near
$\Omega= 2$ in powers of $a$. Since $\Omega= 2$ is a double (semi-simple) eigenvalue of the
bifurcation equation (\ref{bif-prob}) at $a = 0$, we use the two-parameter
perturbation theory with the Taylor expansion
\begin{align}
\label{perturbation-1}\left\{
\begin{array}
[c]{lcl}%
V_{m_{0}+1} & = & c_{m_{0}+1} e_{m_{0}+1,0} + a^{2} \tilde{V}_{m_{0}+1} +
\mathcal{O}_{L^2_r}(a^{4}),\\
W_{-m_{0}+1} & = & c_{-m_{0}+1} e_{m_{0}-1,0} + a^{2} \tilde{W}_{-m_{0}+1} +
\mathcal{O}_{L^2_r}(a^{4}),\\
\Omega & = & 2 + a^{2} \tilde{\Omega} + \mathcal{O}(a^{4}),
\end{array}
\right.
\end{align}
where $(c_{m_0+1},c_{-m_0+1}) \neq (0,0)$ are to be determined,
the correction terms $\tilde{V}_{m_{0}+1}$, $\tilde{W}_{-m_{0}+1}$, and
$\tilde{\Omega}$ are $a$-independent, and the reminder terms
are uniquely defined by the Lyapunov--Schmidt reductions.
The admissible values of $(c_{m_0+1},c_{-m_0+1}) \neq (0,0)$ and $\tilde{\Omega
}$ are found from the matrix eigenvalue problem
\begin{align*}
\tilde{A} \left[
\begin{array}
[c]{c}%
c_{m_{0}+1}\\
c_{-m_{0}+1}%
\end{array}
\right]  = \tilde{\Omega} \left[
\begin{array}
[c]{c}%
c_{m_{0}+1}\\
c_{-m_{0}+1}%
\end{array}
\right]  ,
\end{align*}
where
\begin{align*}
\tilde{A}  &  = \left[
\begin{array}
[c]{cc}%
\langle(-\omega_{m_{0},0} + 2 e_{m_{0},0}^{2}) e_{m_{0}+1,0}, e_{m_{0}+1,0}
\rangle_{L^{2}} & \langle e_{m_{0},0}^{2} e_{m_{0}-1,0}, e_{m_{0}+1,0}
\rangle_{L^{2}}\\
-\langle e_{m_{0},0}^{2} e_{m_{0}+1,0}, e_{m_{0}-1,0} \rangle_{L^{2}} &
-\langle(-\omega_{m_{0},0} + 2 e_{m_{0},0}^{2}) e_{m_{0}-1,0}, e_{m_{0}-1,0}
\rangle_{L^{2}}%
\end{array}
\right] \\
&  = \left[
\begin{array}
[c]{cc}%
\frac{(2m_{0})!}{4^{m_{0}} (m_{0}-1)! (m_{0}+1)!} & \frac{(2m_{0})!}{4^{m_{0}}
m_{0}! \sqrt{(m_{0}-1)! (m_{0}+1)!}}\\
-\frac{(2m_{0})!}{4^{m_{0}} m_{0}! \sqrt{(m_{0}-1)! (m_{0}+1)!}} &
-\frac{(2m_{0})!}{4^{m_{0}} (m_{0}!)^{2}}%
\end{array}
\right]  ,
\end{align*}
and we have used the explicit formula (\ref{HGfunctions}). Eigenvalues of $\tilde{A}$
and their normalized eigenvectors are given by
\begin{align}
\label{eigenvalues-1}\tilde{\Omega} = 0 : \quad\left[
\begin{array}
[c]{c}%
c_{m_{0}+1}\\
c_{-m_{0}+1}%
\end{array}
\right]  = \frac{1}{\sqrt{2m_{0}+1}} \left[
\begin{array}
[c]{c}%
\sqrt{m_{0}+1}\\
-\sqrt{m_{0}}%
\end{array}
\right]
\end{align}
and
\begin{align}
\label{eigenvalues-2}\tilde{\Omega} = \tilde{\Omega}_{m_{0}+1,0} :=
-\frac{(2m_{0})!}{4^{m_{0}} m_{0}! (m_{0}+1)!} : \quad\left[
\begin{array}
[c]{c}%
c_{m_{0}+1}\\
c_{-m_{0}+1}%
\end{array}
\right]  = \frac{1}{\sqrt{2m_{0}+1}} \left[
\begin{array}
[c]{c}%
\sqrt{m_{0}}\\
-\sqrt{m_{0}+1}%
\end{array}
\right]  .
\end{align}
Substituting $\tilde{\Omega} = \tilde{\Omega}_{m_{0}+1,0}$ from
(\ref{eigenvalues-2}) to (\ref{perturbation-1}), we obtain the asymptotic
expansion (\ref{frequency-1}). Since $\tilde{\Omega}_{m_{0}+1,0} < 0$ in
(\ref{eigenvalues-2}), we have $\Omega_{m_{0}+1,0}(a) < 2$ for small $a$.
\end{proof}

\begin{remark}
Lemma \ref{lemma-last-bifurcation} yields the existence of constant $C_{m} >
0$ in item (iii) of Theorem \ref{theorem-main}.
\end{remark}

In order to compute eigenvalues of the blocks $H_{m}(a,\Omega_{m_{0}+1,0}(a))$
for small $a$, we write explicitly the following expansion in
powers of $a$:
\begin{align*}
&  \phantom{t} H_{m}(a,\Omega_{m_{0}+1,0}(a)) = \left[
\begin{array}
[c]{cc}%
- \Delta_{m} + r^{2} - 2 (m+1) & 0\\
0 & - \Delta_{m-2m_{0}} + r^{2} + 2(m-2m_{0}-1)
\end{array}
\right] \\
&  \phantom{t} \phantom{t} + a^{2} \left[
\begin{array}
[c]{cc}%
-\omega_{m_{0},0} - (m-m_{0}) \tilde{\Omega}_{m_{0}+1,0} + 2 e_{m_{0},0}%
^{2}(r) & e_{m_{0},0}^{2}(r)\\
e_{m_{0},0}^{2}(r) & -\omega_{m_{0},0} + (m-m_{0}) \tilde{\Omega}_{m_{0}+1,0}
+ 2e_{m_{0},0}^{2}(r)
\end{array}
\right] \\
&  \phantom{t} \phantom{t} + \mathcal{O}(a^{4}).
\end{align*}
We consider now eigenvalues of $H_{m}(a,\Omega_{m_{0}+1,0}(a))$ denoted by
$\lambda$ near zero as $a \to0$. The following three lemmas summarize the
results of computations of the perturbation theory.

\begin{lemma}
There exists $a_{0} > 0$ such that for every $0 < a < a_{0}$, the
block $H_{m_{0}+1}(a,\Omega_{m_{0}+1,0}(a))$ has a simple zero eigenvalue and
a simple positive eigenvalue of the order $\mathcal{O}(a^{2})$, whereas all
other eigenvalues are strictly positive. \label{lemma-3}
\end{lemma}

\begin{proof}
For $m = m_{0}+1$, computations of the perturbation theory similar to the
expansion (\ref{perturbation-1}) are repeated as follows:
\begin{align}
\label{perturbation-1a}\left\{
\begin{array}
[c]{lcl}%
V_{m_{0}+1} & = & c_{m_{0}+1} e_{m_{0}+1,0} + a^{2} \tilde{V}_{m_{0}+1} +
\mathcal{O}_{L^2_r}(a^{4}),\\
W_{-m_{0}+1} & = & c_{-m_{0}+1} e_{m_{0}-1,0} + a^{2} \tilde{W}_{-m_{0}+1} +
\mathcal{O}_{L^2_r}(a^{4}),\\
\lambda & = & a^{2} \tilde{\lambda} + \mathcal{O}(a^{4}).
\end{array}
\right.
\end{align}
The Lyapunov--Schmidt reduction method results now in the matrix eigenvalue problem
\begin{align*}
\tilde{A} \left[
\begin{array}
[c]{c}%
c_{m_{0}+1}\\
c_{-m_{0}+1}%
\end{array}
\right]  = \tilde{\lambda} \left[
\begin{array}
[c]{c}%
c_{m_{0}+1}\\
c_{-m_{0}+1}%
\end{array}
\right]  ,
\end{align*}
where
\begin{align*}
\tilde{A} = \left[
\begin{array}
[c]{cc}%
\frac{(2m_{0})!}{4^{m_{0}} (m_{0}!)^{2}} & \frac{(2m_{0})!}{4^{m_{0}} m_{0}!
\sqrt{(m_{0}-1)! (m_{0}+1)!}}\\
\frac{(2m_{0})!}{4^{m_{0}} m_{0}! \sqrt{(m_{0}-1)! (m_{0}+1)!}} &
\frac{(2m_{0})!}{4^{m_{0}} (m_{0}-1)! (m_{0}+1)!}%
\end{array}
\right]  .
\end{align*}
Eigenvalues of $\mathcal{A}$ and their normalized eigenvectors are given by
\begin{align}
\label{eigenvalues-1a}\tilde{\lambda} = 0 : \quad\left[
\begin{array}
[c]{c}%
c_{m_{0}+1}\\
c_{-m_{0}+1}%
\end{array}
\right]  = \frac{1}{\sqrt{2m_{0}+1}} \left[
\begin{array}
[c]{c}%
\sqrt{m_{0}}\\
-\sqrt{m_{0}+1}%
\end{array}
\right]
\end{align}
and
\begin{align}
\label{eigenvalues-2a}\tilde{\lambda} = \frac{(2m_{0}+1)!}{4^{m_{0}} m_{0}!
(m_{0}+1)!} : \quad\left[
\begin{array}
[c]{c}%
c_{m_{0}+1}\\
c_{-m_{0}+1}%
\end{array}
\right]  = \frac{1}{\sqrt{2m_{0}+1}} \left[
\begin{array}
[c]{c}%
\sqrt{m_{0}+1}\\
\sqrt{m_{0}}%
\end{array}
\right]  .
\end{align}
The zero eigenvalue in (\ref{eigenvalues-1a}) corresponds to the choice
$\Omega= \Omega_{m_{0}+1,0}(a)$ at the bifurcation point. The positive
eigenvalue in (\ref{eigenvalues-2a}) gives the positive eigenvalue
of the order $\mathcal{O}(a^{2})$ in (\ref{perturbation-1a}). The other eigenvalues of
$H_{m_{0}}(0,2)$ are strictly positive and they remain so in $H_{m_{0}%
+1}(a,\Omega_{m_{0}+1,0}(a))$ for small $a$.
\end{proof}

\begin{lemma}
There exists $a_{0} > 0$ such that for every $0 < a < a_{0}$, the
block $H_{m}(a,\Omega_{m_{0}+1,0}(a))$ with $m \geq2m_{0} + 1$ has a simple
positive eigenvalue of the order $\mathcal{O}(a^{2})$, whereas all other
eigenvalues are strictly positive. \label{lemma-4}
\end{lemma}

\begin{proof}
For $m \geq2m_{0}+1$, the zero eigenvalue of $H_{m}(0,2)$ is simple and all
other eigenvalues are strictly positive. The one-parameter perturbation
expansion for the small eigenvalue is developed as follows:
\begin{align}
\label{perturbation-1b}\left\{
\begin{array}
[c]{lcl}%
V_{m} & = & c_{m} e_{m,0} + a^{2} \tilde{V}_{m} + \mathcal{O}_{L^2_r}(a^{4}),\\
W_{m-2m_{0}} & = & a^{2} \tilde{W}_{m-2m_{0}} + \mathcal{O}_{L^2_r}(a^{4}),\\
\lambda & = & a^{2} \tilde{\lambda} + \mathcal{O}(a^{4}).
\end{array}
\right.
\end{align}
The projection condition yields the only eigenvalue given by
\begin{align}
\tilde{\lambda}  &  = 2 \langle e_{m_{0},0}^{2} e_{m,0}, e_{m,0}
\rangle_{L^2_r} -\omega_{m_{0},0} - (m-m_{0}) \tilde{\Omega}_{m_{0},0}\nonumber\\
&  = \frac{2 (m_{0}+m)!}{2^{m_{0}+m} m_{0}! m!} + \frac{(2m_{0})!
(m-2m_{0}-1)}{4^{m_{0}} m_{0}! (m_{0}+1)!} > 0. \label{eigenvalues-1b}%
\end{align}
Since $\tilde{\lambda} > 0$, the expansion (\ref{perturbation-1b}) yields the
positive eigenvalue of the order $\mathcal{O}(a^{2})$ in the block
$H_{m}(a,\Omega_{m_{0}+1,0}(a))$ for small $a$.
\end{proof}

\vspace{0.25cm}

It remains to consider the blocks $H_{m}(a,\Omega_{m_{0}+1,0}(a))$
for $m_{0}+2 \leq m \leq2m_{0}$. Before continuing with the technical details,
we note the example of $m_{0} = 2$. The results of \cite{Kap1,Kollar} imply
that no real eigenvalues exist in the neighborhood of $\Omega= 2$ and $a = 0$
among eigenvalues of the bifurcation equation (\ref{bif-prob}) for $m = 4 =
2m_{0}$. This is due to oscillatory instability of the radially symmetric
vortex of charge two $(m_{0}=2)$, which arises in the small-amplitude limit of
the primary branch. See Remark 6.9 in \cite{Kap1}. More general results were
obtained in \cite{Hani}, see Proposition 8.3, where all vortices with $m_{0}
\geq2$ were found unstable but the number of unstable modes is smaller 
than $m_{0}-1$ if $m_0$ is sufficiently large. The following result is in
agreement with the outcomes of the stability computations in \cite{Hani,Kap1}.

\begin{lemma}
Let $2 \leq m_{0} \leq16$. There exists $a_{0} > 0$ such that for every $0 < a
< a_{0}$, the block $H_{m}(a,\Omega_{m_{0}+1,0}(a))$ with $m_{0}+2
\leq m \leq2m_{0}$ has two small eigenvalues of the order $\mathcal{O}(a^{2})$
(one is positive and the other one is negative), whereas all other eigenvalues
are strictly positive. \label{lemma-5}
\end{lemma}

\begin{proof}
For $m_{0}+2 \leq m \leq2m_{0}$, the zero eigenvalue of $H_{m}(0,2)$ is double
and all other eigenvalues are strictly positive. The two-parameter
perturbation expansion for the small eigenvalue is developed as follows:
\begin{align}
\label{perturbation-1c}\left\{
\begin{array}
[c]{lcl}%
V_{n} & = & c_{m} e_{m,0} + a^{2} \tilde{V}_{m} + \mathcal{O}_{L^2_r}(a^{4}),\\
W_{m-2m_{0}} & = & c_{m-2m_{0}} e_{2m_{0}-m,0} + a^{2} \tilde{W}_{m-2m_{0}} +
\mathcal{O}_{L^2_r}(a^{4}),\\
\lambda & = & a^{2} \tilde{\lambda} + \mathcal{O}(a^{4}).
\end{array}
\right.
\end{align}
The Lyapunov--Schmidt reduction method results in the matrix eigenvalue problem
\begin{align}
\label{matrix-A-tilde}
\tilde{A} \left[
\begin{array}
[c]{c}%
c_{m}\\
c_{m-2m_{0}}%
\end{array}
\right]  = \tilde{\lambda} \left[
\begin{array}
[c]{c}%
c_{m}\\
c_{m-2m_{0}}%
\end{array}
\right]  ,
\end{align}
where
\begin{align*}
\tilde{A} = \left[
\begin{array}
[c]{cc}%
\frac{2 (m_{0}+m)!}{2^{m_{0}+m} m_{0}! m!} + \frac{(2m_{0})! (m-2m_{0}%
-1)}{4^{m_{0}} m_{0}! (m_{0}+1)!} & \frac{(2m_{0})!}{4^{m_{0}} m_{0}! \sqrt{m!
(2m_{0}-m)!}}\\
\frac{(2m_{0})!}{4^{m_{0}} m_{0}! \sqrt{m! (2m_{0}-m)!}} & \frac{2
(3m_{0}-m)!}{2^{3m_{0}-m} m_{0}! (2m_{0}-m)!} - \frac{(2m_{0})! (m+1)}%
{4^{m_{0}} m_{0}! (m_{0}+1)!}%
\end{array}
\right]  .
\end{align*}
For $m_{0} = 2$ (with $m = 4$) and $m_{0} = 3$ (with $m = 5,6$),
the entries of $\tilde{A}$ are computed explicitly. Since the first diagonal entry is
positive and the second diagonal entry is negative, $\tilde{A}$ has
one positive and one negative eigenvalue $\tilde{\lambda}$. We have checked
numerically that this property remains true for every $2 \leq m_{0} \leq16$.
The expansion (\ref{perturbation-1c}) yields one positive and one negative
eigenvalue of the order $\mathcal{O}(a^{2})$ in the block
$H_{m}(a,\Omega_{m_{0}+1,0}(a))$ for small $a$.
\end{proof}

\begin{remark}
For $m_{0} \geq17$, the matrix $\tilde{A}$ for some $m$ in the range $m_{0}+2
\leq m \leq2m_{0}$ has two negative eigenvalues and the number of such
$m$-values grows with the number $m_{0}$. No zero eigenvalues of $\tilde{A}$
are found numerically for at least $m_{0} \leq100$.
\label{remark-computations-17}
\end{remark}

The following proposition summarizes the previous computations of the
perturbation theory. The corresponding result is needed for the bifurcation
analysis in Section 4.

\begin{proposition}
For every integer $1\leq m_{0}\leq16$, there exists $a_{0}>0$, $C_{m_{0}}%
\in(0,|\tilde{\Omega}_{m_{0+1},0}|)$, and $E_{m_{0}}>0$ such that for every
$0<a<a_{0}$, $|\Omega-\Omega_{m_{0}+1,0}(a)|< C_{m_{0}}a^{2}$, and every
$m\geq m_{0}+2$, the operator $H_{m}(a,\Omega)$ is invertible in $L_{r}%
^{2}(\mathbb{R}^{+})$ with the bound
\begin{equation}
\Vert H_{m}(a,\Omega)^{-1}\Vert_{L_{r}^{2}\rightarrow H_{r}^{2}\cap
L_{r}^{2,2}}\leq E_{m_{0}}a^{-2},\quad m\geq m_{0}+2. \label{bound-loss-a-2}%
\end{equation}
Moreover, all eigenvalues of $H_{m}(a,\Omega)$ are strictly positive, except
for $m_{0}-1$ simple negative eigenvalues, which correspond to $m_{0}+2\leq m\leq2m_{0}$. On
the other hand, all eigenvalues of $H_{m_{0}+1}(a,\Omega)$ are strictly
positive except one simple eigenvalue, which is negative for $\Omega
\lesssim\Omega_{m_{0}+1,0}(a)$ and positive for $\Omega\gtrsim\Omega
_{m_{0}+1,0}(a)$. \label{proposition-last-bifurcation}
\end{proposition}

\begin{proof}
Eigenvalues and invertibility of $H_{m}(a,\Omega_{m_{0}+1,0}(a))$ for $m \geq
m_{0} + 2$ with the bound (\ref{bound-loss-a-2}) follows from the outcomes of
the perturbation theory in Lemmas \ref{lemma-3}, \ref{lemma-4}, and
\ref{lemma-5}, where the $m$-independent constant $E_{m_{0}}$ exists thanks to
the fact that the $\mathcal{O}(a^{2})$ positive eigenvalue in
(\ref{eigenvalues-1b}) is bounded away from zero.

It remains to prove that the zero eigenvalue of $H_{m_{0}+1}(a,\Omega
_{m_{0}+1,0}(a))$ becomes a small positive eigenvalue of $H_{m_{0}+1}%
(a,\Omega)$ for $\Omega\gtrsim\Omega_{m_{0}+1,0}(a)$ and a small negative
eigenvalue of $H_{m_{0}+1}(a,\Omega)$ for $\Omega\lesssim\Omega_{m_{0}%
+1,0}(a)$. This follows from the derivative (\ref{derivative-H}) and the Krein
signature of the zero eigenvalue of $H_{m_{0}+1}(a,\Omega_{m_{0}+1,0}(a))$
defined by (\ref{krein-sign}). We obtain from the expansion
(\ref{perturbation-1a})
\begin{equation}
\label{Krein-m0-1}S_{m_{0}+1} = \| V_{m_{0}+1} \|_{L^{2}_{r}}^{2} - \|
W_{-m_{0}+1} \|^{2}_{L^{2}_{r}} = c_{m_{0}+1}^{2} - c_{-m_{0}+1}^{2} +
\mathcal{O}(a^{2}),
\end{equation}
where $(c_{m_{0}+1},c_{-m_{0}+1})$ is given by the eigenvector of $\tilde{A}$
that corresponds to $\tilde{\lambda} = 0$. From (\ref{eigenvalues-1a}), we
obtain $S_{m_{0}+1} < 0$, hence the corresponding eigenvalue of $H_{m_{0}%
+1}(a,\Omega)$ is an increasing\footnote{If $(c_{m_{0}+1},c_{-m_{0}+1})$ in
(\ref{Krein-m0-1}) is given by the other eigenvector of $\tilde{A}$ that
corresponds to $\tilde{\lambda} > 0$, then it follows from
(\ref{eigenvalues-2a}) that $S_{m_{0}+1} > 0$. Hence, the corresponding small
positive eigenvalue of $H_{m_{0}+1}(a,\Omega)$ is a decreasing function of
$\Omega$. Nevertheless, for $\Omega\gtrsim\Omega
_{m_{0}+1,0}(a)$, these two small eigenvalues of $H_{m_{0}+1}(a,\Omega)$ are
ranged in the same order of $\mathcal{O}(a^{2})$ as at $\Omega= \Omega
_{m_{0}+1,0}(a)$.} function of $\Omega$.
\end{proof}

\begin{remark}
Propositions \ref{proposition-secondary-bifurcations} and
\ref{proposition-last-bifurcation} complete the proof of item (iii) in Theorem
\ref{theorem-main}.
\end{remark}

\begin{remark}
Eigenvalues $\lambda := \Omega (m-m_0)$ of the bifurcation problem
(\ref{bif-prob}) near $\lambda_{m,m_0} = 2(m-m_0)$ are either complex or real for $m_{0}+2 \leq m \leq2m_{0}$,
depending on whether the $m$-th mode of the $m_{0}$-th vortex is spectrally
unstable or stable. When all such eigenvalues are complex, which happens for
$1 \leq m_{0} \leq3$, no other bifurcation curve is connected to the point
$\Omega_{0} = 2$ from below, besides the curve $\Omega_{m_{0}+1,0}$. When
$m_{0} \geq4$, we have found that there are $R(m_{0})$ pairs of real eigenvalues
$\lambda$ of the bifurcation problem (\ref{bif-prob}) near $\lambda_{m,m_0}$,
e.g.
\begin{itemize}
\item $R(4) = 1$ with $m = 8$;

\item $R(5) = 1$ with $m = 10$;

\item $R(6) = 2$ with $m = 11,12$;

\item $R(7) = 3$ with $m = 12,13,14$;

\item $R(8) = 3$ with $m = 14,15,16$;
\end{itemize}
and so on. This finding corresponds to the result of Proposition 8.3 in \cite{Hani} 
where the number of complex eigenvalues is found to be smaller than $m_0 - 1$ 
if $m_0$ is sufficiently large. If $R(m_0) \neq 0$, 
then there exist $R(m_{0})$ bifurcation curves connected
to the point $\Omega_{0} = 2$ from below. As follows from the count on
negative eigenvalues in
\[
N(m_{0}) - B(m_{0}) - 1 = m_{0} - 1,
\]
which coincides with the number of negative eigenvalues in Lemma
\ref{lemma-5}, these additional bifurcation curves for $4 \leq m_{0} \leq16$
are located above the curve $\Omega_{m_{0}+1,0}$. However, for $m_{0} \geq17$, 
thanks to the computations in Remark \ref{remark-computations-17}, 
some of the positive eigenvalues of $H_{m}(a,\Omega)$ for $m_{0} + 2 \leq m \leq2 m_{0}$ become negative
eigenvalues for $\Omega\lesssim\Omega_{m_{0}+1,0}(a)$ and the total number of
negative eigenvalues at $\Omega\gtrsim\Omega_{m_{0}+1,0}(a)$ exceeds $m_{0} -
1$. Therefore, some of the $R(m_{0})$ bifurcation curves are located below the curve
$\Omega_{m_{0}+1,0}$ for $m_0 \geq 17$.
\end{remark}

\section{Secondary branches of multi-vortex solutions}

Recall that the solution $U$ to the stationary GP equation (\ref{EcStat}) is a
critical point of the energy functional $E_{\mu}(u)$ in (\ref{Lyapunov}),
therefore, the bifurcation problem for $g(v;a,\Omega)$ in
(\ref{g-function})--(\ref{operator-g}) has a variational structure. The number of negative
eigenvalues of the Jacobian operator $\mathcal{H(}a,\Omega)$ in
(\ref{Hessian-definition})--(\ref{Hessian}) (which is known as \emph{the Morse index})
changes at every bifurcation curve as $\Omega$ crosses
$\Omega_{m,n}(a)$, according to Propositions
\ref{proposition-secondary-bifurcations} and
\ref{proposition-last-bifurcation}, where the values of $\Omega_{m,n}(a)$ are
given by Lemmas \ref{lemma-count-B} and \ref{lemma-last-bifurcation}, see
equations (\ref{bifurcation-point}) and (\ref{frequency-1}).

Here we prove that for each fixed $a$ and for each non-resonant bifurcation
point, there is a continuous branch of solutions of $g(v;a,\Omega)$
bifurcating from $(0;a,\Omega_{m,n}(a))$ on one side of the bifurcation point
$\Omega=\Omega_{m,n}(a)$. The new family of multi-vortex solutions
is parameterized by two parameters $(a,\Omega)$.

Besides proving the local bifurcation result, we discuss symmetries of the
bifurcating branches and their global continuation with respect to parameter
$\Omega$. For definitions and methods used to prove the equivariant bifurcation we refer to \cite{BaKr10,GoSc86,IzVi03}.

In section 4.1, symmetries of $g(v;a,\Omega)$, in particular, its
equivariant properties are analyzed. In section 4.2, we prove the local
bifurcation result for a non-resonant bifurcation point $\Omega_{m,n}(a)$,
with a simple zero eigenvalue of $\mathcal{H(}a,\Omega)$. We also discuss
symmetries and asymptotic estimates of the bifurcating branches, which are
needed to study the location of the individual vortices in the multi-vortex
configurations. In section 4.3, we prove the global continuation of the
solution branches.

\subsection{Symmetries and equivariance of $g(v;a,\Omega)$}

We define the action of the group $O(2)=S^{1}\cup\kappa S^{1}$ by
\begin{equation}
\rho(\varphi)v(r,\theta)=e^{-im_{0}\varphi}v(r,\theta+\varphi),\qquad
\rho(\kappa)v(r,\theta)=\bar{v}(r,-\theta)\text{.} \label{v}%
\end{equation}
The operator $g(v;a,\Omega)$ given by (\ref{g-function})--(\ref{operator-g}) is $O(2)$-equivariant
by the action of the group given by (\ref{v}). That is, we have $g(\rho
(\varphi)v)=\rho(\varphi)g(v)$ since
\begin{align*}
e^{im_{0}\varphi}g(\rho(\varphi)v)(r,\theta-\varphi)  &  =\left(  -\omega
_{m}(a) -\Delta_{(r,\theta)} + r^{2} + \Omega i(\partial_{\theta}%
-im_{0})\right)  v(r,\theta) -e^{im_{0}\theta}\psi_{m_{0}}(r;a)^{3}\\
&  +\left\vert e^{im_{0} \theta}\psi_{m_{0}}(r;a)+v(r,\theta)\right\vert ^{2}
\left(  e^{im_{0} \theta} \psi_{m_{0}}(r;a) +v(r,\theta)\right)  .
\end{align*}
Similarly, we have $g(\rho(\kappa)v)=\rho(\kappa)g(v)$.

As is explained in Section 2.2, the component $v$ is extended to the vector
$\mathbf{v}=(v,w)$ with the constraint $w=\bar{v}$, so that the root finding
problem is formulated for the analytic nonlinear operator $\mathbf{g}%
(\mathbf{v})=(g(v,w),\bar{g}(v,w))$. The natural extension of the action of
the group $O(2)$ to the second component of ${\bf v} = (v,w)$ is
\begin{equation}
\rho(\varphi)w(r,\theta)=e^{im_{0}\varphi}w(r,\theta+\varphi),\qquad
\rho(\kappa)w(r,\theta)=\bar{w}(r,-\theta)\text{.} \label{w}%
\end{equation}
In the Fourier basis
\[
v=\sum_{m\in\mathbb{Z}}V_{m}(r)e^{im\theta},\qquad w=\sum_{m\in\mathbb{Z}%
}W_{m}(r)e^{im\theta},
\]
the action of the group $O(2)=S^{1}\cup\kappa S^{1}$ is given by
\begin{align*}
\rho(\varphi)V_{m}  &  =e^{i(m-m_{0})\varphi}V_{m},\qquad\rho(\kappa
)V_{m}=\bar{V}_{m},\\
\rho(\varphi)W_{m}  &  =e^{i(m+m_{0})\varphi}W_{m},\qquad\rho(\kappa
)W_{m}=\bar{W}_{m}.
\end{align*}
so that
\begin{align}
\rho(\varphi)(V_{m},W_{m-2m_{0}})  &  =e^{i(m+m_{0})\varphi}(V_{m}%
,W_{m-2m_{0}}),\label{action}\\
\rho(\kappa)(V_{m},W_{m-2m_{0}})  &  =(\bar{V}_{m},\bar{W}_{m-2m_{0}}%
)\text{.}\nonumber
\end{align}
Therefore, the subspaces of functions $(V_{m},W_{m-2m_{0}})$ are composed of
similar irreducible representations under the action of the group $O(2)$.

The subspace $(V_{m},W_{m-2m_{0}})$ has as isotropy group, the dihedral group
$D_{m-m_{0}}$ generated by the elements $\kappa$ and $\zeta=2\pi/(m-m_{0})$.
The dihedral group $D_{m-m_{0}}$ will be used to find the symmetry-breaking
bifurcations of the primary branch into the multi-vortex solutions along the secondary
branches. Due to the symmetries of $D_{m-m_0}$, the multi-vortex solution
is represented by a $(m-m_0)$-polygon of individual vortices.

For a fixed value of $m\in\mathbb{Z}$, the action of $\rho
(\zeta)$ is given by%
\[
\rho(\zeta)(V_{j},W_{j-2m_{0}})=\exp\left(  2\pi i\frac{j-m_{0}}{m-m_{0}%
}\right)  (V_{j},W_{j-2m_{0}}),\quad j\in\mathbb{Z}.
\]
The fixed point space
\[
\mathrm{Fix}(D_{m-m_{0}})=\{(v,w)\in L^{2}(\mathbb{R}^{2}):\quad\rho
(\gamma)(v,w)=(v,w)\text{ for }\gamma\in D_{m-m_{0}}\}
\]
is composed of functions with real components $(V_{j},W_{j-2m_{0}})$ such that
$j-m_{0}$ is a multiple of $m-m_{0}$. If $(v,\bar{v})\in\mathrm{Fix}%
(D_{m-m_{0}})$, then $v$ can be characterized by
\[
v(r,\theta)=\sum_{j\in m_{0}+(m-m_{0})\mathbb{Z}}V_{j}(r)e^{ij\theta
}=e^{im_{0}\theta}\sum_{j\in(m-m_{0})\mathbb{Z}}V_{m_{0}+j}(r)e^{ij\theta},
\]
where all functions $\{V_{j}(r)\}_{j\in m_{0}+(m-m_{0})\mathbb{Z}}$ are
real-valued. Writing $v(r,\theta)=e^{im_{0}\theta}\phi(r,\theta)$, we deduce
that $\phi$ satisfies the symmetry constraints:
\begin{equation}
\phi(r,\theta)=\bar{\phi}(r,-\theta)=\phi(r,\theta+\zeta)\text{.} \label{sym}%
\end{equation}

Since $\mathbf{g}$ is $O(2)$-equivariant, the operator $\mathbf{g}%
(\mathbf{v})$ restricted to $\mathrm{Fix}(D_{m-m_{0}})$ is well defined.
Therefore, we can consider the bifurcation problem
\begin{equation}
\mathbf{g}^{D_{m-m_{0}}}(\mathbf{v};a,\Omega):\;X\cap\mathrm{Fix}(D_{m-m_{0}%
})\times\mathbb{R}\times\mathbb{R}\rightarrow\mathrm{Fix}(D_{m-m_{0}}),
\label{bifurcation-problem-D}%
\end{equation}
where $X:=H^{2}(\mathbb{R}^{2})\cap L^{2,2}(\mathbb{R}^{2})$ is the graph norm
of the Jacobian operator $\mathcal{H}$. A schematic
illustration of the local bifurcations of the primary and secondary branches is given on
Figure \ref{fig-2}.

\begin{figure}[th]
\centering

\begin{center}
\resizebox{15.0cm}{!}{ \includegraphics{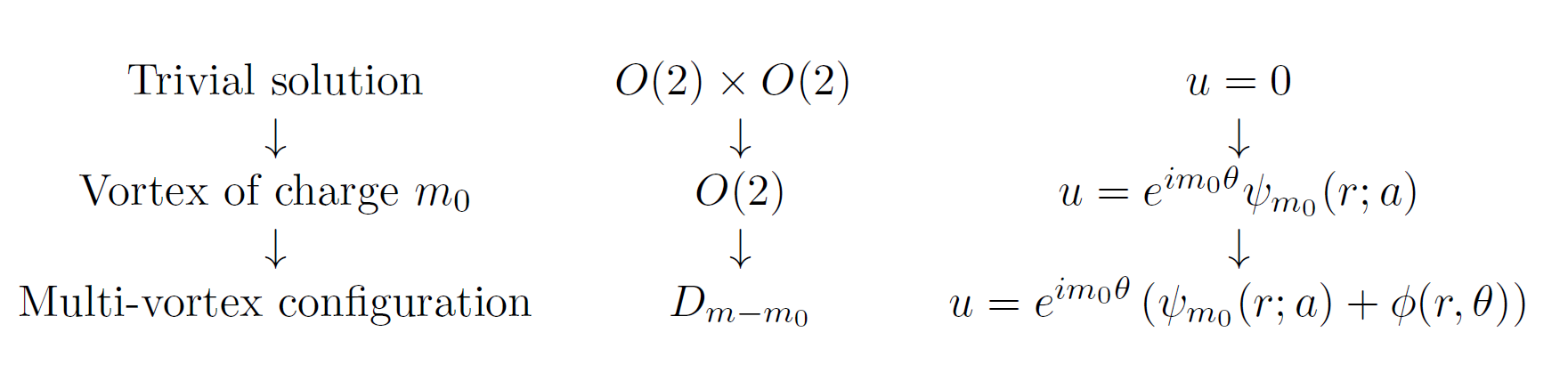}}
\end{center}

\caption{The isotropy lattice for the symmetry-breaking bifurcations.}%
\label{fig-2}
\end{figure}

By Schur's lemma, the Jacobian operator $\mathcal{H}$ for $\mathbf{g}$ has a
diagonal decomposition in the subspaces of similar irreducible representations
given by the components $(V_{m},W_{m-2m_{0}})$. Indeed, this has been done in
(\ref{Hessian}) and (\ref{Hessian-blocks}), where the operator $\mathcal{H}$
in the subspace $(V_{m},W_{m-2m_{0}})$ is represented by the block $H_{m}$.
Consequently, for a fixed $m\in\mathbb{Z}$, the Jacobian operator of
$\mathbf{g}^{D_{m-m_{0}}}$ consists of the blocks $H_{j}$ corresponding to
$j-m_{0}\in(m-m_{0})\mathbb{Z}$. Moreover, in the subspace $\mathrm{Fix}%
(D_{m-m_{0}})$ we have $w=\bar{v}$, so that $W_{j-m_{0}}=V_{-\left(
j-m_{0}\right)  }$ and the blocks $H_{j}$ with negative $j-m_{0}$ are
determined by those with $j-m_{0}\in\mathbb{N}$. Hence, we denote
\[
\mathcal{H}^{D_{m-m_{0}}} = \mathrm{diag}\{H_{j}\}_{j\in m_{0}+(m-m_{0})\mathbb{N}}.
\]

The operator $\mathcal{H}$ has a zero eigenvalue in the block $j=m_{0}$ due to
the gauge invariance of the original problem. This zero eigenvalue is not
present for the operator $\mathcal{H}^{D_{m-m_{0}}}$ in $\mathrm{Fix}%
(D_{m-m_{0}})$ because the reflection $\kappa\in D_{m-m_{0}}$ excludes the
gauge invariance. Furthermore, the double eigenvalues of $\mathcal{H}$ in the
blocks with positive and negative $j-m_{0}$ become the simple eigenvalues of
$\mathcal{H}^{D_{m-m_{0}}}$ in $\mathrm{Fix}(D_{m-m_{0}})$ again due to the
reflection $\kappa$.

\subsection{Local bifurcation results}

Here we prove a local bifurcation from a simple eigenvalue of $\mathcal{H}%
^{D_{m-m_{0}}}$ that exists at $\Omega= \Omega_{m,n}(a)$ for small $a$,
according to (\ref{bifurcation-point}) and (\ref{frequency-1}) in Lemmas
\ref{lemma-count-B} and \ref{lemma-last-bifurcation}. The restriction of the
space $X$ to the fixed-point space $\mathrm{Fix}(D_{m-m_{0}})$ is useful in
two aspects. First, it allows us to prove the local bifurcation from a simple
eigenvalue by avoiding resonances from the components that are not contained
in $\mathrm{Fix}(D_{m-m_{0}})$. Second, it gives additional information on
symmetries of the bifurcating solutions $\mathbf{v}$. The symmetries are useful to understand the
distributions of individual vortices in the $(m-m_{0})$-polygons.

The local bifurcation results are obtained for the non-resonance bifurcation
points, according to the following definition. This definition extends
Definition \ref{remark-non-resonance}.

\begin{definition}
For a fixed $a > 0$, we say that $\Omega_{m,n}(a)\in(0,2)$ is a non-resonant
bifurcation point if the kernel of $\mathcal{H}^{D_{m-m_{0}}}(\Omega
_{m,n}(a))$ has dimension one. We say that $\Omega_{m,n} \in(0,2)$ is a
non-resonant curve if this condition holds for each small $a$.
\end{definition}

For each curve $\Omega_{m,n}$, the non-resonant condition is given by the
following equivalent conditions:

\begin{itemize}
\item[(i)] $H_{j}(a,\Omega_{m,n}(a))$ is invertible;

\item[(ii)] $\Omega_{m,n}(a) \neq\Omega_{j,k}(a)$;

\item[(iii)] $\mu_{j,k}^{-}(a,\Omega_{m,n}(a)) \neq0$;
\end{itemize}
\noindent where $j$ takes values in $m_{0}+(m-m_{0})\ell$, $\ell\in\mathbb{N}
\backslash\{1\}$, $k \in \mathbb{N}_0$, and $a>0$ is arbitrary but sufficiently small.

As we discussed in Remark \ref{remark-resonant-curve}, the bifurcation curves
are all non-resonant for $1\leq m_{0}\leq3$ and the first resonance happens
for $m_{0}=4$ because $\Omega_{10,0}(0)=\Omega_{7,1}(0)=2/3$. In view of the
restriction on the range of $j$ in the space $\mathrm{Fix}(D_{m-m_{0}})$,
however, the bifurcation curve $\Omega_{10,0}$ is non-resonant because
$\mathcal{H}^{D_{6}}$ is composed of blocks $H_{j}$ with $j=4,10,16,...$ and
the zero eigenvalue $\mu_{7,1}^{-}(0,\Omega_{10,0}(0))$ is not included in the
spectrum of $\mathcal{H}^{D_{6}}$. On the other hand, the bifurcation curve
$\Omega_{7,1}$ may be resonant because $\mathcal{H}^{D_{3}}$ is composed of
blocks $H_{j}$ with $j=4,7,10,...$ and the zero eigenvalue $\mu_{10,0}%
^{-}(0,\Omega_{7,1}(0))$ is included in the spectrum of $\mathcal{H}^{D_{3}}$.
To know exactly if $\Omega_{7,1}(a)$ is resonant with $\Omega_{10,0}(a)$ one
needs to compute the normal form in $a$, which is out of the scope of our presentation.

\begin{remark}
The curves $\Omega_{m,0}$ for $2m_{0}+1\leq m\leq3m_{0}-1$ are non-resonant.
Even if the resonance occurs in $\mathcal{H}$, e.g. for $m_{0}=4$, it does not
show up in $\mathcal{H}^{D_{m-m_{0}}}$. Indeed, if $\Omega_{m,0}(0)=\Omega_{j,k}(0)$
for $j-m_{0}=(m-m_{0})\ell$ with $\ell\in\mathbb{N}\backslash\{1\}$, then
$$
f(m) = f(j) - \frac{2k}{j-m_0},
$$
where
$$
f(j) = \frac{m_{0}-\left\vert j-2m_{0}\right\vert }{j-m_{0}} = \left\{
\begin{array}[c]{cc} 1 & m_{0} < j \leq 2 m_{0}\\
\frac{2m_{0}}{j-m_{0}} - 1 & 2m_{0} < j \end{array} \right.
$$
Note that $f$ is a strictly decreasing function on $[2m_0,\infty)$. If $k=0$, then $j > 2m_{0}$, hence $f(m)=f(j)$
is true only if $j=m$ ($\ell = 1$), which is excluded.
If $k \geq 1$, then $f(m) < f(j)$, which implies that $m > j$ or $j-m_{0}<m-m_{0}$.
Therefore, the possible resonant block
$H_{j}$ with $m_{0} < j < m$ is not in $\mathcal{H}^{D_{m-m_{0}}} = {\rm diag}(H_{m_{0}},H_{m},H_{2m-m_0},\ldots)$.
\end{remark}

\begin{remark}
\label{remark-last-curve} The curve $\Omega_{m_{0}+1,0}(a) = 2 +
\mathcal{O}(a^{2})$ is non-resonant as long as the matrices $\tilde{A}$
arising in the matrix eigenvalue problem (\ref{matrix-A-tilde}) are invertible.
We have checked this condition numerically for $1 \leq m_{0} \leq100$.
\end{remark}

The following proposition follows from the Crandall-Rabinowitz theorem, see
Theorem I.5.1 in \cite{HKiel}. It covers the non-resonant bifurcation curve
$\Omega_{m,n}$, for which Proposition \ref{proposition-secondary-bifurcations}
applies. It does not cover the curve $\Omega_{m_{0}+1,0}$ in Remark
\ref{remark-last-curve}.

\begin{proposition}
\label{ThmLoc} For each non-resonant curve $\Omega_{m,n}\in(0,2)$
parameterized by $a>0$ sufficiently small, the operator $\mathbf{g}%
^{D_{m-m_{0}}}(\mathbf{v};a,\Omega)$ in (\ref{bifurcation-problem-D}) admits a
new family of roots $\mathbf{v} \in\mathrm{Fix}(D_{m-m_{0}})$ and $\Omega \in (0,2)$ parameterized by real $b$ such that
\begin{equation}
\Omega(a,b)=\Omega_{m,n}(a)+\mathcal{O}(b^{2}) \label{new-root-Omega}%
\end{equation}
and
\begin{equation}
v(r,\theta;a,\Omega(a,b)) = b f_{m,n}(r,\theta;a) + \mathcal{O}_X(a b^{2},
b^{3}), \label{new-root-v}%
\end{equation}
where
\begin{equation}
f_{m,n}(r,\theta;a) = e_{|m-2m_{0}|,n}(r)e^{i(2m_{0}-m)\theta} + \mathcal{O}_X(a^{2}) \label{f-}%
\end{equation}
is the eigenvector of $H_{m}(a,\Omega_{m,n}(a))$ associated with the zero
eigenvalue $\mu_{m,n}^{-}(a,\Omega_{m,n}(a))$.
\end{proposition}

\begin{proof}
The local bifurcation problem (\ref{bifurcation-problem-D}) is well-defined
for the operator $\mathbf{g}^{D_{m-m_{0}}}$. The operator $\mathbf{g}%
^{D_{m-m_{0}}}$ has a linearization given by $\mathcal{H}^{D_{m-m_{0}}}$ and
its kernel is spanned by the eigenvector $f_{m,n}$ associated to the
simple zero eigenvalue $\mu_{m,n}^{-}(a;\Omega_{m,n}(a))$ under the assumption
of the proposition. Since $\mathcal{H}^{D_{m-m_{0}}}$ has a uniformly bounded
inverse operator in the complement of the kernel, according to Proposition
\ref{proposition-secondary-bifurcations}, we are in the position to define the
bifurcation equation as in Theorem I.5.1 in \cite{HKiel}.

The only condition to be verify is that
\[
\partial_{\Omega}\mathcal{H}^{D_{m-m_{0}}}f_{m,n} = i\partial_{\theta} f_{m,n}
\]
is not in the range of $\mathcal{H}^{D_{m-m_{0}}}$. Thanks to the basis in
(\ref{basis-bifurcation}) and the fact that the zero eigenvalue corresponds to
$\mu_{m,n}^{-}$, the leading-order approximation of the eigenvector
$f_{m,n}$ is given by (\ref{f-}) for $a>0$ sufficiently small. Then,
$i \partial_{\theta} f_{m,n} \notin {\rm Ran}(\mathcal{H}^{D_{m-m_{0}}})$
because
\[
\left\langle i\partial_{\theta}f_{m,n},f_{m,n}\right\rangle_{L^2} = m -
2m_{0} +\mathcal{O}(a^{2})\neq0\text{.}%
\]

The existence of the new root of $\mathbf{g}^{D_{m-m_{0}}}(v;a,\Omega)$ and
the estimate (\ref{new-root-v}) for $a>0$ sufficiently small follow from the
Crandall-Rabinowitz theorem, where the scaling $\mathcal{O}_X(ab^2)$ is due to the cubic terms in the expressions
for $g$ in (\ref{operator-g}). This theorem gives also the estimate $\Omega(a,b)
= \Omega_{m,n}(a)+\mathcal{O}(b)$. Furthermore, the $S^{1}$-action
(\ref{action}) of the element $\varphi=\pi/(m+m_{0})$ in the kernel generated
by $f_{m,n}$ is given by $\rho(\varphi)=-1$. Therefore, the bifurcation
equation is odd and $\partial_{vv}\mathbf{g}^{D_{m-m_{0}}}(0)(f_{m,n},f_{m,n}) = 0$. The estimate (\ref{new-root-Omega}) is obtained
from formula (I.6.3) in \cite{HKiel}.
\end{proof}

\begin{remark}
The new family (\ref{new-root-Omega}) and (\ref{new-root-v}) exists on one
side of the bifurcation curve $\Omega_{m,n}$, that is,
\[
\Omega(a,b)=\Omega_{m,n}(a)+cb^{2}+\mathcal{O}(b^{4}),
\]
where $c$ can be computed from (I.6.11) in \cite{HKiel}. If $c>0$ the
bifurcation is supercritical pitchfork (to the right of the bifurcation curve)
and the Jacobian operator at the new (secondary) branch of solutions has one
more negative eigenvalue compared to that at the primary branch. If $c<0$ the
bifurcation is subcritical pitchfork (to the left of the bifurcation curve)
and the Jacobian operator at the new branch of solutions has one less negative
eigenvalue compared to that at the primary branch. Because the new family can
be rotated in the $(x,y)$ plane, the Jacobian operator at the new branch has
an additional zero eigenvalue related to this rotation symmetry.
\end{remark}

The following proposition covers the non-resonant bifurcation curve
$\Omega_{m+1,0}$, for which Proposition \ref{proposition-last-bifurcation} applies.

\begin{proposition}
\label{ThmLast} For the non-resonant curve $\Omega_{m_{0}+1,0}$ parameterized
by $a > 0$ sufficiently small, the operator $\mathbf{g}^{D_{1}}(\mathbf{v}%
;a,\Omega)$ in (\ref{bifurcation-problem-D}) admits a new family of roots
$\mathbf{v} \in \mathrm{Fix}(D_{1})$ and $\Omega \in (0,2)$ parameterized by
real $b$ such that
\begin{equation}
\label{last-root-Omega}\Omega(a,b) = \Omega_{m_{0}+1,0}(a) + \mathcal{O}(a^{2}
b^{2})
\end{equation}
and
\begin{equation}
\label{last-root-v}v(r,\theta;a,\Omega(a,b)) = a \left[  b f_{m_{0}+1,0}(r,\theta;a) + \mathcal{O}_X(b^{2}) \right]  ,
\end{equation}
where $f_{m_{0}+1,0}$ is the eigenvector
of $H_{m_{0}+1}(a,\Omega_{m_{0}+1,0}(a))$ associated with the zero eigenvalue.
\end{proposition}

\begin{proof}
The scaling of $a$ in (\ref{last-root-v}) is needed due to the loss of
$\mathcal{O}(a^{-2})$ in the bound (\ref{bound-loss-a-2}) on the inverse
operator $(\mathcal{H}^{D_{1}})^{-1}$, according to Proposition
\ref{proposition-last-bifurcation}. Since $\psi_{m_{0}}=\mathcal{O}(a)$, the
nonlinear terms in the operator $\mathbf{g}^{D_{1}}(\mathbf{v};a,\Omega)$ are
now scaled by $\mathcal{O}(a^{3})$, hence the loss of $\mathcal{O}(a^{-2})$
produces the terms of the expansion (\ref{last-root-v}) at the order
$\mathcal{O}(a)$ and higher. Hence the bifurcation problem is closed at the
order $\mathcal{O}(a)$ and the proof follows the one in Proposition
\ref{ThmLoc} with a new parameter $b$, which is scaled independently of $a$.
\end{proof}

\begin{remark}
The local bifurcation in Proposition \ref{ThmLast} is also of the pitchfork
type. Thanks to the computations in Lemma \ref{lemma-3}, the
leading-order approximation of the eigenvector $f_{m_{0}+1,0}$ is given
by
\begin{equation}
\label{ff}
f_{m_0+1,0}(r,\theta;a)=c_{m_{0}+1}e_{m_{0}+1,0}(r)e^{i(m_{0}+1)\theta
}+c_{-m_{0}+1} e_{m_{0}-1,0}(r)e^{i(m_{0}-1)\theta}+\mathcal{O}_X(a^{2}),
\end{equation}
for $a>0$ sufficiently small, where $(c_{m_{0}+1},c_{1-m_{0}})$ is an
eigenvector of the matrix $\tilde{A}$ computed in (\ref{eigenvalues-1a}).
\end{remark}

\begin{remark}
Propositions \ref{ThmLoc} and \ref{ThmLast} yield the proof of item (iv) in Theorem \ref{theorem-main}.
\end{remark}

\subsection{Global bifurcation}

We obtain the global bifurcation result in the fixed-point space $\mathrm{Fix}%
(D_{m-m_{0}})$ by using the topological degree theory in the case of
simple eigenvalues. It is usually referred to as the global Rabinowitz result,
see Theorem 3.4.1 of \cite{Ni2001}. The global bifurcation result means that
the solution branch $(v,\Omega)$ that originates at the non-resonant
bifurcation curve $(0,\Omega_{m,n})$ either reaches the boundaries $\Omega=0$
and $\Omega=2$ or return to another bifurcation point $(0,\Omega_{\ast})$ or
diverges to infinite values of $v$ for a finite value of $\Omega\in
\lbrack0,2)$.

The following result holds because the Jacobian operator $\mathcal{H}$ given by
(\ref{Hessian-definition}) and (\ref{Hessian}) is bounded and has closed range
for $\left\vert \Omega\right\vert <2$.

\begin{lemma}
\label{Lem1} Let $X=H^{2}(\mathbb{R}^{2})\cap L^{2,2}(\mathbb{R}^{2})$ be the
domain space for the Jacobian operator $\mathcal{H}$. For every $\left\vert
\Omega\right\vert <2$ there is a positive constant $c$ such that the operator
$\left(  \mathcal{H}+cI\right)  :X\rightarrow L^{2}(\mathbb{R}^{2})$ is
positive definite and $\left(  \mathcal{H}+cI\right)  ^{-1}:X\rightarrow X$ is compact.
\end{lemma}

\begin{proof}
The eigenvalues of $\mathcal{H}$ are given by $\mu_{m,n}^{\pm}(a,\Omega)$
expanded as in (\ref{mu-plus-minus}). For $\left\vert \Omega\right\vert <2$,
the eigenvalues $\mu_{m,n}^{\pm}$ are bounded from below and do not accumulate at a finite value.
Therefore, there is a
positive constant $c$ such that the bounded operator $\mathcal{H+}cI : X \to L^2(\mathbb{R}^2)$
is positive definite and invertible.  Since $X$ is compactly included
in $L^{2}(\mathbb{R}^{2})$, then the inverse operator $\left(  \mathcal{H}%
+cI\right)  ^{-1}:X\hookrightarrow L^{2}\rightarrow X$ is compact.
\end{proof}

\begin{remark}
Observe in (\ref{mu-plus-minus}) that for $\left\vert \Omega\right\vert =2$
the eigenvalues $\mu_{m,n}^{\pm}(a,\Omega)$ can accumulate at a finite value
as $a \to0$, while for $\left\vert \Omega\right\vert >2$ the eigenvalues
$\mu_{m,n}^{\pm}(a,\Omega)$ are unbounded both from above and from below. As a
result, the operator $\mathcal{H}+cI: X \rightarrow L^{2}(\mathbb{R}^{2})$
does not have a closed range for $|\Omega| = 2$, its
inverse $\left(  \mathcal{H}+cI\right)  ^{-1}: L^{2}(\mathbb{R}^{2})
\rightarrow X$ is not bounded, and the inverse operator $\left(
\mathcal{H}+cI\right)  ^{-1} : X \to X$ is not compact.
\end{remark}

By Lemma \ref{Lem1}, the Jacobian operator $\mathcal{H}$ is Fredholm of the
degree zero for $\Omega\in[0,2)$. Also the restricted operator $\mathcal{H}%
^{D_{m-m_{0}}}(\Omega)$ is a self-adjoint Fredholm operator for every
$\Omega\in\lbrack0,2)$. Since $\mathcal{H}^{D_{m-m_{0}}}(\Omega)$ is
invertible for $\Omega$ close but different from the non-resonant bifurcation
curve $\Omega_{m,n}$, then the Morse index $n^{D_{m-m_{0}}}(\Omega)$ of
$\mathcal{H}^{D_{m-m_{0}}}(\Omega)$ restricted to $\ker\mathcal{H}%
^{D_{m-m_{0}}}(\Omega_{m,n})$ for $\Omega$ close to $\Omega_{m,n}$ is well
defined. Let $\eta^{D_{m-m_{0}}}(\Omega_{m,n})$ be the net crossing number of
eigenvalues of $\mathcal{H}^{D_{m-m_{0}}}(\Omega)$ defined by
\begin{equation}
\eta^{D_{m-m_{0}}}(\Omega_{m,n}):=\lim_{\varepsilon\rightarrow0}\left|
n^{D_{m-m_{0}}}(\Omega_{m,n}+\varepsilon)-n^{D_{m-m_{0}}}(\Omega
_{m,n}-\varepsilon)\right|. \label{index-def}%
\end{equation}
If $\Omega_{m,n}$ is a non-resonant bifurcation curve, then it is obvious that
$\mathcal{\eta}^{D_{m,n}}(\Omega_{m,n})=1$. The following proposition gives
the global bifurcation result for each non-resonant bifurcation curve.

\begin{proposition}
Fix $a>0$ sufficiently small, if $\eta^{D_{m-m_{0}}}(\Omega_{m,n})$ is odd for
$\Omega_{m,n}\in\lbrack0,2)$, the nonlinear operator $g(v;a,\Omega)$ has a
global bifurcation of solutions $(v,\Omega)$ in $\mathrm{Fix}(D_{m_{-}m_{0}%
})\times\lbrack0,2)$ arising from $(v,\Omega)=(0,\Omega_{m,n})$.
\label{Theorem-global}
\end{proposition}

\begin{proof}
Since $X$ is a Banach algebra with respect to pointwise multiplication and
$g(0;a,\Omega)=0$, we obtain the expansion
\[
g(v;a,\Omega)=\mathcal{H(}\Omega)v+\mathcal{O}_X(v^2).
\]
We can apply the global Rabinowtz theorem to the nonlinear operator
\[
f(v,\Omega)=\left(  \mathcal{H}+cI\right)  ^{-1}g(v;a,\Omega)v = Iv -
c \left(\mathcal{H}+cI\right)^{-1} v + \mathcal{O}_X(v^2),
\]
where $c>0$ is defined in Lemma \ref{Lem1}. The operator $f$ is also
equivariant and can be restricted to $\mathrm{Fix}(D_{m-m_{0}})$ denoted by
$f^{D_{m-m_{0}}}$. The index for bifurcation of $f$ in $\mathrm{Fix}%
(D_{m-m_{0}})$ is up to an orientation factor, the jump on the local indices
as $\Omega$ crosses $\Omega_{m,n}$. That is, since $\eta^{D_{m-m_{0}}}%
(\Omega_{m,n})$ is odd, then
\begin{align}
&  \deg\left(  \left\Vert x\right\Vert -\varepsilon,f^{D_{m-m_{0}}}%
(x,\Omega);B_{2\varepsilon}\times B_{2\rho}\right) \nonumber\\
&  =\deg(f^{D_{m-m_{0}}}(x,\Omega-\rho);B_{2\varepsilon})-\deg(f^{D_{m-m_{0}}%
}(x,\Omega+\rho);B_{2\varepsilon})\nonumber\\
&  =\pm\left(  1-\left(  -1\right)  ^{\eta^{D_{m-m_{0}}}}\right)
=\pm2\text{,} \label{BIndex}%
\end{align}
where $B_{2\varepsilon}$ and $B_{2\rho}$ are ball of radius $2\varepsilon$ and
$2\rho$ around $0\in X\cap\mathrm{Fix}(D_{m-m_{0}})$ and $\Omega_{m,n}%
\in\lbrack0,2)$, respectively.
\end{proof}

\begin{remark}
If the branch from $(0,\Omega_{m,n})$ returns to another bifurcation point
$(0,\Omega_{m^{\prime},n^{\prime}})$, then the sum of all the bifurcation
indices (\ref{BIndex}) at the bifurcation points has to be equal zero.
Therefore, the knowledge of the exact factor $\pm$ in (\ref{BIndex}) is
helpful to obtain information of where the branches can return. The exact
factor $\pm$ in (\ref{BIndex}) can be computed for all the bifurcation curves
using the fact that
\[
\deg(f^{D}(x,\Omega);B_{2\varepsilon})=(-1)^{n^{D_{m-m_0}}(\Omega)}\text{,}%
\]
since $\mathcal{H+}cI$ is positive definite. For example, for the last bifurcation from
$\Omega_{m_{0}+1,0}$ with $1 \leq m_{0}\leq16$, the exact index is
\[
\deg\left(  \left\Vert x\right\Vert -\varepsilon,f^{D_{1}}(x,\Omega
);B_{2\varepsilon}\times B_{2\rho}\right)  =(-1)^{m_{0}}-(-1)^{m_{0}%
-1}=(-1)^{m_{0}}2\text{.}%
\]
Therefore, this branch can return to a single bifurcation point $\Omega_{0}$
only if the latter point has index $-2 (-1)^{m_{0}}$.
\end{remark}

\begin{remark}
Proposition \ref{Theorem-global} provides a proof of the claim in item (iii) of Theorem \ref{theorem-main}
that the bifurcations in the interval $[a^2 D_{m_0},2-a^2 C_{m_0}]$ are global.
\end{remark}

\section{Individual vortices in the multi-vortex configurations}

We can assume $a>0$ in the expansions (\ref{ua}) and (\ref{branch-1}) for the primary branch after a
change of phase. Also, we can choose the sign of $b$ by a shift of $\theta$,
i.e. we can assume $b>0$ in the expansions (\ref{new-root-v}) and
(\ref{last-root-v}) for the secondary branch. Here we analyze the location of
individual vortices in the multi-vortex configurations bifurcating along the
secondary branch.

First, we prove that the total vortex charge is preserved near the origin when
the secondary branch bifurcates off from the primary branch.

\begin{lemma}
\label{Preservation}Fix $R>0$. There exists $b_{0}>0$ such
that the degree of the bifurcating solution $U$ along the secondary
branch on the circle of the radius $R$ is $m_{0}$ for every $b\in
\lbrack0,b_{0})$.
\end{lemma}

\begin{proof}
We recall that $a>0$ and $\psi_{m_{0}}(r)>0$ for every $r\in(0,\infty)$. For
every fixed $R>0$, there exists a sufficiently small $b_{0}>0$ such that the
bifurcating solution $U(r,\theta)$ in Propositions
\ref{proposition-secondary-bifurcations} and
\ref{proposition-last-bifurcation} is nonzero at $r=R$ for every $b\in
\lbrack0,b_{0})$. This follows from the smallness of the
error terms  in the expansions (\ref{new-root-v}) and
(\ref{last-root-v}) in the norm of $X=H^{2}(\mathbb{R}^{2})\cap L^{2,2}(\mathbb{R}^{2})$, which is
embedded in $C^{0}(\mathbb{R}^{2})$. Since $U(r,\theta)$ is nonzero at $r=R$,
the degree of $U$ on the disk $B_R$ of radius $R$
is well defined and does not change for every
$b\in\lbrack0,b_{0})$. Since the degree is $m_{0}$ at $b=0$, it
remains $m_{0}$ for every $b\in\lbrack0,b_{0})$.
\end{proof}

\begin{remark}
Because $\psi_{m_{0}}(r)\rightarrow0$ as $r\rightarrow\infty$, 
we are not able to claim that additional zeros of $U(r,\theta)$
cannot come from infinity as $b\neq0$. If such zeros exist, additional
individual vortices come from infinity on a very small background 
$U(r,\theta)$.
\end{remark}

Next, we rewrite the eigenfunctions $e_{m,n}(r)$ of the linear eigenvalue
problem (\ref{def.vmn}) in the form
\begin{equation}
e_{m,n}(r)=p_{m,n}(r)e^{-r^{2}/2}, \label{Hermite-polynomials}%
\end{equation}
where $p_{m,n}(r)$ is a polynomial of degree $\left\vert m\right\vert +2n$, which is chosen
to be positive for $r$ near zero. The first eigenfunctions $e_{m,0}(r)$ and
$e_{m,1}(r)$ in (\ref{HGfunctions}) and (\ref{H-1-functions}) are given by
(\ref{Hermite-polynomials}) with
\begin{equation}
p_{m,0}(r)=\frac{\sqrt{2}}{\sqrt{m!}}r^{m},\quad p_{m,1}(r)=\frac{\sqrt{2}%
}{\sqrt{(m+1)!}}r^{m}(m+1-r^{2}). \label{Hermite-poly-first}%
\end{equation}

The following proposition deals with the secondary bifurcations described in Proposition
\ref{ThmLoc}.

\begin{proposition}
Let $0<b\lesssim a$ and consider the bifurcating solution to the stationary GP
equation (\ref{EcStat}) in the form (\ref{bif}) given by the expansions
(\ref{ua}) and (\ref{new-root-v}). Let $r_{0}$ be the first positive zero of
the function
\begin{equation}
z(r):=ap_{m_{0},0}(r)-bp_{|m-2m_{0}|,n}(r) \label{first-zero}%
\end{equation}
and assume that it is a simple zero\footnote{The assumption is always
satisfied if $0 < b \ll a$ since $p_{|m-2m_{0}|,n}(r)$ is
positive for small $r$ and $|m-2m_{0}| < m_{0}$.}. Then, the bifurcating
solution has simple zeros arranged in the $\left(  m-m_{0}\right)  $-polygon on
a circle of radius $\rho$ with $\rho=r_{0}+\mathcal{O}(a^{2})$.
\label{proposition-zeros}
\end{proposition}

\begin{proof}
By combining (\ref{ua}), (\ref{bif}), (\ref{new-root-v}), and (\ref{f-}), we
obtain an asymptotic representation of the bifurcating solutions $U$ in the
form
\[
U(r,\theta)=ae_{m_{0},0}(r)e^{im_{0}\theta}+be_{|m-2m_{0}|,n}(r)e^{i\left(
2m_{0}-m\right)  \theta}+\mathcal{O}_X(a^{3},a^{2}b,ab^{2},b^{3}).
\]
Zeros of $U(r,\theta)$ are equivalent to the zeros of
\begin{equation}
e^{-im_{0}\theta}U(r,\theta)=ae_{m_{0},0}(r)+be_{|m-2m_{0}|,n}(r)e^{-i(m-m_{0}%
)\theta}+\phi(r,\theta), \label{zeros}%
\end{equation}
where $\phi = \mathcal{O}_X(a^{3},a^{2}b,ab^{2},b^{3})$ satisfies the
symmetry constraints (\ref{sym}).

The function $e^{-i\left(  m-m_{0}\right)  \theta}$ is real only if
$\theta=k\zeta$ and $\theta=(k+1/2)\zeta$, where $\zeta=2\pi/\left(
m-m_{0}\right)  $ and $k\in\mathbb{Z}$. For these angles, the function
$\phi(\theta,r)$ is real by the symmetries (\ref{sym}). Therefore, the
function (\ref{zeros}) is real if and only if $\theta=k\zeta$ and
$\theta=(k+1/2)\zeta$. These two choices of 
angles give two choices of the $(m-m_{0})$-polygons of zeros along
a circle of radius $\rho$.

To determine the small radius $\rho$ in the limit $b\rightarrow0$, we
factorize the factor $e^{-r^{2}/2}$ in the eigenfunctions
(\ref{Hermite-polynomials}) and truncate the error term $\phi$.
Since we assume that $p_{m,n}(r)$ is
positive for $r$ near zero, then the right-hand side of (\ref{zeros})
is strictly positive for $\theta=k\zeta$ and $r \gtrsim 0$.
For $\theta=(k+1/2)\zeta$, we have $e^{-i\left(  m-m_{0}\right)  \theta}=-1$, 
hence the right-hand side of (\ref{zeros}) has a zero only if $z(r)$ 
in (\ref{first-zero}) has a positive root.

Let $r_{0}$ be the first positive root of $z$ in (\ref{first-zero}) and assume that it is
simple. Since the function $\phi(r,\theta)$ is small in the norm of
$X=H^{2}(\mathbb{R}^{2})\cap L^{2,2}(\mathbb{R}^{2})$ by Proposition
\ref{ThmLoc}, an application of the implicit function theorem proves that the
representation (\ref{zeros}) with $\theta=\zeta/2=\pi/(m-m_{0})$ has the $(m-m_{0})$ polygon of simple zeros at
the circle of radius $\rho$, where $\rho=r_{0}+\mathcal{O}(a^{2})$.
\end{proof}

\begin{remark}
For the last bifurcation with $m=m_{0}+1$ described in Proposition
\ref{ThmLast}, a similar result cannot be proven because the small parameter
$a$ is scaled out from the expansion (\ref{last-root-v}).
The remainder term $\phi = \mathcal{O}_X(a b^2)$ in the representation for $U(r,\theta)$ may give a contribution
to the distribution of individual vortices, which is comparable with the leading-order
term $a \psi_{m_0}(r) e^{i m_0 \theta}$ and the bifurcating mode
$a b f_{m_0+1,0}(r,\theta;a)$.
\label{remark-last}
\end{remark}

In the rest of this section, we study individual vortices in the bifurcating multi-vortex configurations.

\subsection{$(m-m_{0})$-polygons of vortices}

Polygons made of vortices rotating at a constant speed have been studied 
for many models: fluids, BECs and superconductors. It has been
found that these relative equilibria of $m$ vortices are stable for $m\leq7$,
see, e.g., \cite{GaIz12,theo1} and references therein. We have found that
similar multi-vortex configurations appear along the secondary branches
bifurcating from the primary branch of the radially symmetric vortex of charge $m_{0}\geq2$. As an
example, we give precise information about the vortex polygons in the
particular cases $n=0$ and $n=1$. For $n=1$, the bifurcation is similar to the
bifurcation of complex multi-vortex solutions described in Lemma 3.3 of
\cite{Kap1} for $m_{0}=6$.

%-----------------------------------------------------------------------
\begin{figure}[p]
\begin{center}
\resizebox{15.0cm}{!}{
\includegraphics{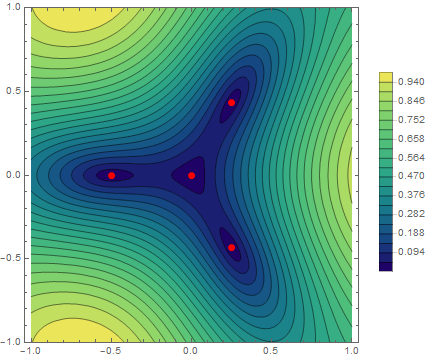} \qquad
\includegraphics{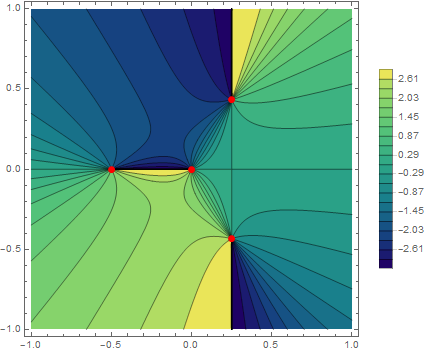}} \resizebox{15.0cm}{!}{
\includegraphics{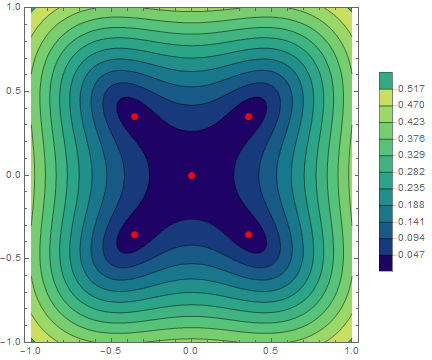} \qquad
\includegraphics{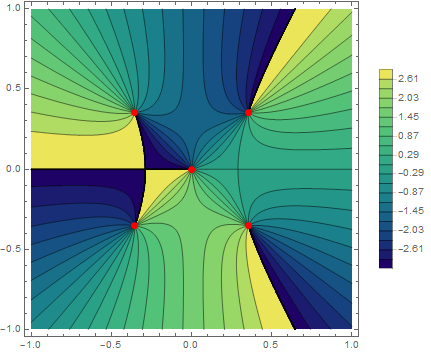}} \resizebox{15.0cm}{!}{
\includegraphics{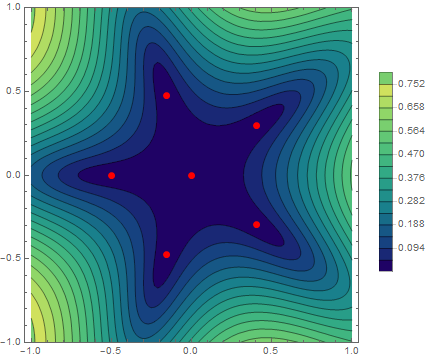} \qquad
\includegraphics{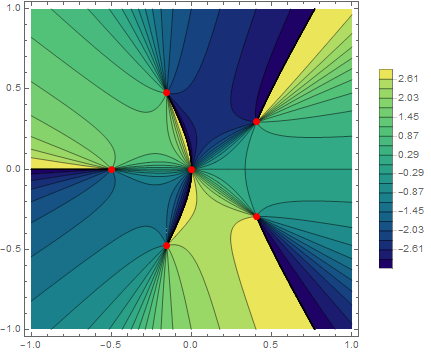}}
\end{center}
\caption{ The left and right columns illustrate the norm and phase of the
truncated solution $U$ in  (\ref{expansion-U}). Top: $m_{0}=2$ near $\Omega_{5,0}$.
Middle: $m_{0}=3$ near $\Omega_{7,0}$. Bottom: $m_{0}=3$ near $\Omega_{8,0}$. }%
\label{fig-3}%
\end{figure}
%-----------------------------------------------------------------------

\subsubsection{Case $n=0$: Vortex polygons with a central vortex}

Bifurcation occurs at the bifurcation curve $\Omega_{m,n} \in (0,2)$ with $n=0$ and
$2m_{0}+1\leq m\leq3m_{0}-1$ (when $m_{0}\geq2$) in accordance with Lemma
\ref{lemma-count-B}, Propositions \ref{proposition-secondary-bifurcations},
\ref{ThmLoc}, and \ref{proposition-zeros}. By using
(\ref{Hermite-poly-first}) and (\ref{first-zero}), we write explicitly
\begin{align*}
z(r)  &  =ap_{m_{0},0}(r)-bp_{m-2m_{0},0}(r)\\
&  =\frac{\sqrt{2}}{\sqrt{m_{0}!}}r^{m-2m_{0}}\left(  ar^{3m_{0}-m}%
-b\frac{\sqrt{m_{0}!}}{\sqrt{(m-2m_{0})!}}\right)  ,
\end{align*}
where we recall that $2m_{0}<m<3m_{0}$. If $0<b\lesssim a$, the first positive
zero of $z(r)$ is located at
\begin{equation}
\label{root-r-0}
r_{0}=\left(  \frac{b}{a}\frac{\sqrt{m_{0}!}}{\sqrt{(m-2m_{0})!}}\right)
^{1/(3m_{0}-m)}.
\end{equation}
By Proposition \ref{proposition-zeros}, we have a $(m-m_{0})$-polygon of
simple zeros of the function
\begin{equation}
\label{expansion-U}
U(r,\theta) = \left[  a p_{m_{0},0}(r) + b p_{m-2m_{0},0}(r) e^{-i(m-m_{0}%
)\theta} \right]  e^{-r^{2}/2} e^{im_{0}\theta} + \phi(r,\theta)
e^{im_{0}\theta},
\end{equation}
at points $(r,\theta) = (\rho,\zeta/2 + k \zeta)$, where $\zeta= 2\pi/(m-m_{0})$,
$k \in\{0,1,...,m-m_{0}-1\}$, $\rho= r_{0} + \mathcal{O}(a^{2})$, and $r_{0}$
is given by (\ref{root-r-0}).

We claim that the degree of each simple zero of $U(r,\theta)$ is $+1$,
which means that each zero of $U$ on the $(m-m_{0})$-polygon represents a
vortex of charge one. By symmetry of $D_{m-m_{0}}$, each zero in the
$(m-m_{0})$-polygon has equal degree, hence it is sufficient to
compute the degree at the simple zero $(r,\theta)=(\rho,\zeta/2)$.
Using Taylor expansion of $U$ in (\ref{expansion-U}) for $b\lesssim a$, we obtain
\[
cU(r,\theta)=z^{\prime}(\rho)(r-\rho)+i b (m-m_{0}) p_{m-2m_{0},0}(\rho
)(\theta-\zeta/2)+\mathcal{O}(2),
\]
where $c\in\mathbb{C}$ is constant and $\mathcal{O}(2)$ denotes quadratic
remainder terms of the Taylor expansion. Because $m-2m_{0}<m_{0}$, we have
$z(r)<0$ for $r>0$ sufficiently small, therefore, $z^{\prime}(\rho)>0$ for
$b>0$ sufficiently small. On the other hand, $m > 2m_{0}$ and $p_{m-2m_{0}%
,0}(\rho)>0$ in the same limit. Therefore, the degree of $U(r,\theta)$
at $(r,\theta)=(\rho,\zeta/2)$ is $+1$.

In addition, $U(r,\theta)$ in  (\ref{expansion-U})  has a zero at $r=0$ if the remainder term
$e^{im_{0}\theta}\phi(r,\theta)$ is truncated. Let $d$ be the degree of $U$ in a
neighborhood of $r=0$. The degree in the disk $B_R$ of a
sufficiently large radius $R$ is equal to sum of the local degrees in the
disk. By Lemma \ref{Preservation}, we have $d+m-m_{0}=m_{0}$, hence
$d=2m_{0}-m<0$.

When the remainder term $e^{im_{0}\theta}\phi(r,\theta)$ is taken
into account in  (\ref{expansion-U}), the multiple zero of $U$ at $r=0$ may split from the origin.
However, by the symmetry in $D_{m-m_{0}}$, if
the central vortex splits, then it breaks into $m-m_{0}$ vortices of equal
charge $|d|/(m-m_{0})$. Since $m-2m_{0}\leq m_0 - 1 < m_0 + 1 \leq m-m_{0}$,
then $|d|/(m-m_0) < 1$ and the central vortex never splits.

\begin{remark}
For the case $m_{0}=2$, we have the bifurcation point $\Omega_{5,0}=2/3+\mathcal{O}(a^{2})$. 
Since $m=5$ and $n=0$, we have a configuration of three vortices of charge one that
form an equilateral triangle and a central vortex of charge $d=2m_{0}-m=-1$
(top panel of Figure \ref{fig-3}).
\end{remark}

\begin{remark}
For the case $m_{0}=3$, we have two bifurcation points $\Omega_{7,0}%
=1+\mathcal{O}(a^{2})$ and $\Omega_{8,0}=2/5+\mathcal{O}(a^{2})$. At the
former bifurcation, the bifurcating branch has four vortices of charge one
that form a square and the central vortex of charge $d=2m_{0}-m=-1$ (middle
panel of Figure \ref{fig-3}). At the latter bifurcation, the bifurcating
branch has five vortices of charge one that form an equilateral pentagon and
the central vortex of charge $d=2m_{0}-m=-2$ (bottom panel of Figure
\ref{fig-3}).
\end{remark}

\subsubsection{Case $n=1$: Vortex polygons without a central vortex}

Bifurcation occurs at the bifurcation curve $\Omega_{m,n}$ with $n=1$ and
$m_{0}+3\leq m\leq3m_{0}-3$ (when $m_{0}\geq3$) in accordance with Lemma
\ref{lemma-count-B}, Propositions \ref{proposition-secondary-bifurcations},
\ref{ThmLoc}, and \ref{proposition-zeros}. By using
(\ref{Hermite-poly-first}) and (\ref{first-zero}), we write explicitly
\begin{align*}
z(r)  &  =ap_{m_{0},0}(r)-bp_{|m-2m_{0}|,1}(r)\\
&  =\frac{\sqrt{2}}{\sqrt{m_{0}!}}r^{|m-2m_{0}|}\left(  ar^{m_{0}-|m-2m_{0}%
|}-b C_{m,m_{0}} (|m-2m_{0}|+1-r^{2})\right)  ,
\end{align*}
where
\[
C_{m,m_{0}} = \frac{\sqrt{m_{0}!}}{\sqrt{(|m-2m_{0}|+1)!}}
\]
and we recall that $|m-2m_0| < m_0$. If $0<b\ll a$, the first positive zero of
$z(r)$ is located at
\[
r_{0}=\left(  \frac{b}{a} C_{m,m_{0}}(|m-2m_{0}|+1)\right)^{1/(m_{0}%
-|m-2m_{0}|)}\left[  1 + \mathcal{O}\left(\left(  \frac{b}{a}\right)^{2/(m_{0}%
-|m-2m_{0}|)}\right) \right].
\]
By Proposition \ref{proposition-zeros}, we have the $(m-m_{0})$ polygon of
vortices on the circle of radius $\rho=r_{0}+\mathcal{O}(a^{2})$. Each vortex
has charge one by using the same arguments as in the case $n=0$.

\begin{remark}
If $m = 2m_0$, the polygon of $m_0$ charge-one vortices surrounds the origin with no
central vortex. For $m_0 = 3$, the bifurcation point is
$\Omega_{6,1} = 2/3 + \mathcal{O}(a^2)$ and the secondary branch has three charge-one vortices
located at the equilateral triangle. For $m_0 = 6$ studied in \cite{Kap1},
the bifurcation point is $\Omega_{12,1} = 4/3 + \mathcal{O}(a^2)$
and the secondary branch has six charge-one vortices at a hexagon (top panel of Figure \ref{fig-4}).
\end{remark}

\begin{remark}
If $m\neq2m_{0}$, $U(r,\theta)$ has zero at $r=0$ if the
remainder term $\phi(r,\theta)$ is truncated. By Lemma \ref{Preservation},
the central zero of $U$ corresponds to the vortex of charge $d=2m_{0}-m$,
where $-m_0 < d < m_0$. The central vortex may split into $m-m_0$ vortices of equal
charge only if $|d|$ is divisible by $m-m_0$.
\end{remark}

\subsection{Asymmetric vortex and asymmetric vortex pair}

\begin{figure}[p]
\begin{center}
\resizebox{15.0cm}{!}{
\includegraphics{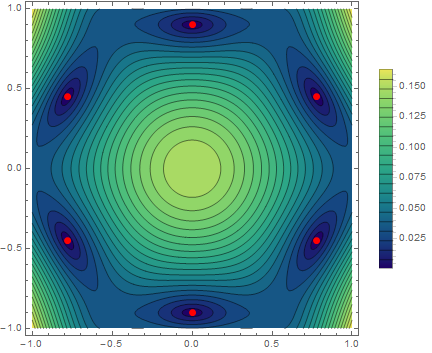} \qquad
\includegraphics{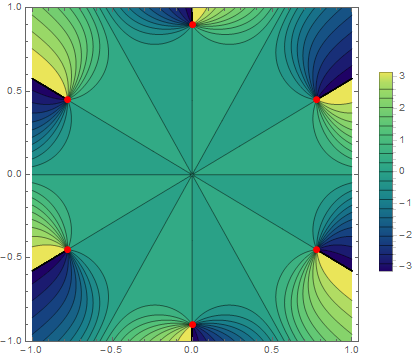}} \resizebox{15.0cm}{!}{
\includegraphics{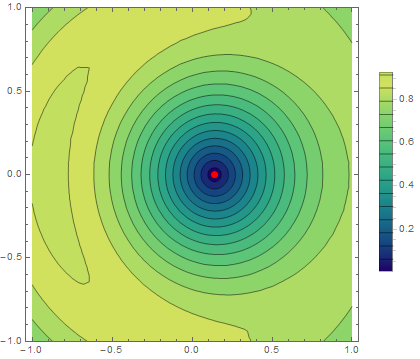} \qquad
\includegraphics{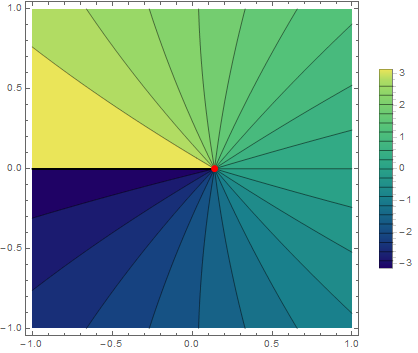}} \resizebox{15.0cm}{!}{
\includegraphics{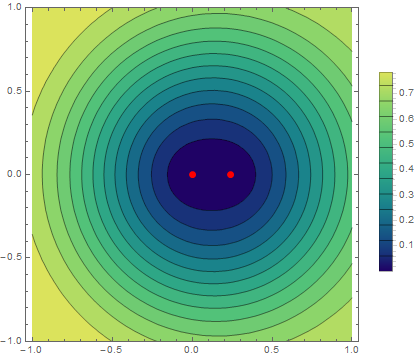} \qquad
\includegraphics{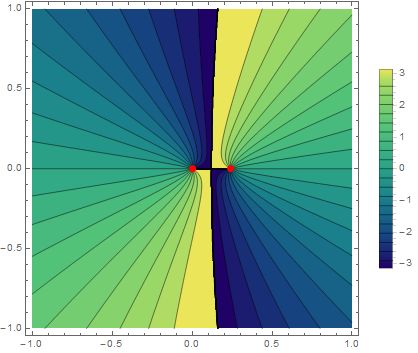}}
\end{center}
\caption{ The left and right columns illustrate the norm and phase of the
truncated solution $U$. Top: $m_{0}=6$ near $\Omega_{12,1}$. Center: $m_{0}=1$ near
$\Omega_{2,0}$. Bottom: $m_{0}=2$ near $\Omega_{3,0}$. }%
\label{fig-4}%
\end{figure}

Bifurcation occurs at the bifurcation curve $\Omega_{m_{0}+1,0}=2+\mathcal{O}%
(a^{2})$ (when $m_{0}\geq1$) in accordance with Lemma
\ref{lemma-last-bifurcation}, Propositions \ref{proposition-last-bifurcation} and
\ref{ThmLast}. By using (\ref{last-root-v}) and (\ref{ff}), we write explicitly
\[
U(r,\theta)=a\left[  p_{m_{0},0}(r)+bc_{m_{0}+1}p_{m_{0}+1,0}(r)e^{i\theta
}+bc_{-m_{0}+1}p_{m_{0}-1,0}(r)e^{-i\theta}\right]  e^{-r^{2}/2}%
e^{im_{0}\theta} + \phi(r,\theta)e^{im_{0}\theta},
\]
where $(c_{m_{0}+1},c_{-m_{0}+1})$ is obtained from the eigenvector in
(\ref{eigenvalues-1a}) and $\phi = \mathcal{O}_X(a b^2)$, see Remark \ref{remark-last}.
In particular, we have $c_{m_{0}+1}>0$ and $c_{-m_{0}+1}<0$.

\begin{remark}
If $m_0 = 1$ and $0< b \lesssim a$, the simple zero of $U(r,\theta)$ near the origin is located
at $\rho = b|c_0|+\mathcal{O}(b^{2})$. The degree of $U$ near the
simple zero at $(r,\theta) = (\rho,\pi)$ is again $+1$, so that the corresponding vortex has charge one.
Since no other zeros of $U(r,\theta)$ are located near the origin,
the bifurcating solution at the secondary branch corresponds to the
asymmetric vortex obtained in \cite{PeKe13} (center panel of Figure \ref{fig-4}).
\end{remark}

\begin{remark}
If $m_{0}=2$ and $0< b \lesssim a$, the double zero of $U(r,\theta)$ at the origin for $b = 0$
split to the distances $\rho_{\pm} = \mathcal{O}(b)$
according to the roots of the quadratic equation
\begin{equation}
\label{quadratic-roots}
r^2 \pm b |c_{-1}| \sqrt{2} + b^2 \beta = 0,
\end{equation}
where $\beta$ is a numerical constant obtained from the remainder term $\phi(r,\theta)$,
whereas the plus and minus signs correspond to the choice $\theta = 0$ and $\theta = \pi$ respectively.
Only positive roots of the quadratic equations (\ref{quadratic-roots})
are counted, and according to Lemma \ref{Preservation},
we should have the total of two positive roots at both sign combinations.
Indeed, if $\beta > 0$, the two positive roots $\rho_{\pm} = \mathcal{O}(b)$ exist for $\theta = \pi$
and no positive roots for $\theta = 0$, while if $\beta < 0$,
one positive root $\rho_+$ exists for $\theta = 0$ and one positive root exists for $\theta = \pi$.
In both cases, $\rho_+ \neq \rho_-$, so that the bifurcating solution at the secondary branch corresponds
to the asymmetric pair of two charge-one vortices obtained in \cite{Navarro} 
(bottom panel of Figure \ref{fig-4} in the case $\beta < 0$).
\end{remark}

\begin{remark}
If $m_0 \geq 3$ and $0< b \lesssim a$, the multiple root of $U(r,\theta)$ at the origin for $b = 0$
split to the distances $\rho_{\pm} = \mathcal{O}(b)$ according to the roots of the $n$-th order
polynomial equation, which is obtained from computations of the remainder term $\phi(r,\theta)$
up to the order of $b^n$. By Lemma \ref{Preservation}, there must exist exactly $n$
roots to the two polynomial equations for $\theta = 0$ and $\theta = \pi$ but the precise
characterization of these roots depend on the coefficients of the polynomial equation.
\end{remark}

\begin{remark}
Proposition \ref{proposition-zeros} and computations in Sections 5.1 and 5.2 yield
the proof of item (v) of Theorem \ref{theorem-main}. All items of Theorem \ref{theorem-main} have been proved.
\end{remark}

\noindent\textbf{Acknowledgement.} The authors are indebted to A. Contreras,
P. Kevrekidis and M. Tejada-Wriedt for discussions related to this project and collaboration.

\end{document}